\documentclass[12pt]
{amsart}
\usepackage{amsfonts,amsmath,amssymb,amsthm,graphicx,afterpage,url}
\usepackage{hyperref}
\usepackage{color}

\input{xy}
\xyoption{all}

\def\id{\mathop{\fam0 id}}

\def\t{\widetilde}
\def\R{{\mathbb R}} \def\Z{{\mathbb Z}}

\long\def\comment#1\endcomment{}

\newcommand{\jonly}[1]{}
\newcommand{\aronly}[1]{#1}

    \theoremstyle{theorem}
         \newtheorem{theorem}{Theorem}
         \newtheorem{lemma}[theorem]{Lemma}
         \newtheorem{corollary}[theorem]{Corollary}
         
    \theoremstyle{definition}
         \newtheorem{remark}[theorem]{Remark}

\begin{document}

\title{On embeddability of joins and their `factors'}


\author{S. Parsa and A. Skopenkov}

\thanks{S. Parsa: DePaul University, Chicago, IL., USA. \texttt{s.parsa@depaul.edu}. A. Skopenkov: \texttt{https://users.mccme.ru/skopenko}.
Independent University of Moscow, Moscow Institute of Physics and Technology, Russian Federation.  \texttt{skopenko@mccme.ru}.
\newline
We are grateful to S. Melikhov and the anonymous referee for useful suggestions.
}


\date{}

\maketitle

\abstract
We present a short and clear proof of the following particular case of a 2006 result of Melikhov-Shchepin.
{\it Let $K$ be a $k$-dimensional simplicial complex and $K*[3]$ the union of three cones over $K$ along their common bases.
If $2d\ge3k+3$ and $K*[3]$ embeds into $\R^{d+2}$, then $K$ embeds into $\R^d$.}
We also present a generalization of this theorem.
The proofs are based on the Haefliger-Weber `configuration spaces' embeddability criterion,
equivariant suspension theorem and simple properties of joins and cones.
\endabstract


\section{Introduction}

The interest in embeddability of complexes into Euclidean space (and in related problems) was recently revived, see survey \cite[\S3.2, \S3.3]{Sk18}, the references therein, and \cite{GS06, Pa15, Sk18o, Pa21, Me22}.
We present a short and clear
\jonly{proof of Theorem \ref{t:parsaj}}
\aronly{proofs of Theorems \ref{t:parsaj} and \ref{t:eq}.b}
which first appeared  in the unpublished paper
\cite[$(iv)\Rightarrow(i)$ of Corollary 4.4\aronly{, Theorem 4.5}]{MS06}.
\jonly{See also \cite[Theorem C]{Me22} and \cite[Theorem 3.b]{PS20}.}
We also prove Theorem \ref{t:eq}\aronly{.a} which is a generalization of Theorem \ref{t:parsaj}\aronly{, and is a version of Theorem \ref{t:eq}.b without any condition on the given embedding of the join $K*L$}.

We abbreviate `$k$-dimensional finite simplicial complex' to `$k$-complex'.

\begin{theorem}\label{t:parsaj}
Let $K$ be a $k$-complex and $K*[3]$ the union of three cones over $K$ along their common bases.
If $2d\ge3k+3$ and $K*[3]$ embeds into $\R^{d+2}$, then $K$ embeds into $\R^d$.
\end{theorem}

This improves in a sense the well-known Flores example that the join $[3]^{*k-1}$ of $k-1$ copies of the 3-point set $[3]$ does not embed into $\R^{2k}$.

\begin{corollary}\label{c:links}
If $2d\ge3k+3$ and a $(k+1)$-complex $P$ embeds into $\R^{d+2}$, then the triple intersection of links of any three vertices of $P$ is a $k$-complex embeddable into $\R^d$.
\end{corollary}

This follows by Theorem \ref{t:parsaj} because $P$ contains the join of $[3]$ and the triple intersection.

Denote by
$$Y^{\times2}_\Delta:=\{(x,y)\in Y\times Y\ :\ x\ne y\}\quad\text{and}
\quad Y^{*2}_\Delta:=\{[(x,y,t)]\in Y*Y\ :\ x\ne y\}$$
the deleted product and the deleted join of a complex $Y$.
Consider the antipodal involution on $S^m$ and the involutions $(x,y)\leftrightarrow(y,x)$ and $[(x,y,t)]\leftrightarrow[(y,x,t)]$ on these spaces.

\begin{theorem}\label{t:eq} Let $K$ and $L$ be a $k$-complex and a complex.
Assume that $2d\ge3k+3$, $g:K*L\to\R^{d+q+1}$ is an embedding and \aronly{either

(a)} there is a $\Z_2$-equivariant map $\varphi:S^q\rightarrow L^{*2}_{\Delta}$\aronly{, or

(b) there is a $\Z_2$-equivariant map $\psi:S^{q-1}\rightarrow L^{\times2}_{\Delta}$ and $g$ is {\it level-preserving}, i.e. $g([x,y,t])\subset\R^{d+q}\times t$ for any $x\in K$, $y\in L$ and $t\in[0,1]$}.


Then $K$ embeds into $\R^d$.
\end{theorem}

\section{Proofs}

Theorem \ref{t:parsaj} follows by Theorem \ref{t:eq}\aronly{.a} or by (parts (a,d,f,b) of) the following Lemma \ref{l:steps} (because for $k=0$ Theorem \ref{t:parsaj} is trivial, and for $k\ge1$ we have $3k+3\ge2(k+2)$).

In the rest of this paper we replace $Y^{\times2}_\Delta$ and $Y^{*2}_\Delta$ by their $\Z_2$-equivariantly homotopy equivalent simplicial versions \cite[\S5]{Ma03}.

Let $X$ be a complex with an involution $\sigma:X\to X$.
Denote by $\pi^m_{\Z_2}(X)$ the set of $\Z_2$-equivariant maps $X\to S^m$.
Consider the involution on
$\Sigma X=\frac{X\times[-1,1]}{X\times\{-1\},X\times1}$ defined by $[(x,t)]\mapsto[(\sigma(x),-t)]$.

\begin{lemma}\label{l:steps} Let $K$ be a $k$-complex.

(a) If $K$ embeds into $\R^d$, then $\pi^{d-1}_{\Z_2}(K^{\times2}_\Delta)\ne\emptyset$.

(b) (Weber theorem) If $\pi^{d-1}_{\Z_2}(K^{\times2}_\Delta)\ne\emptyset$ and $2d\ge3k+3$, then $K$ embeds into $\R^d$.

(c) (equivariant suspension theorem) The equivariant suspension
$$\Sigma:\pi_{\Z_2}^{m-1}(X) \to \pi_{\Z_2}^m(\Sigma X)$$
is a 1--1 correspondence for $\dim X\le2m-4$ and is surjective for $\dim X\le2m-3$.

(d) There is an equivariant surjective map
$q:K^{*2}_\Delta\to \Sigma(K^{\times2}_\Delta)$
whose only non-trivial preimages are those of the vertices of the suspension and are homeomorphic to $K$.

(e) There is a $\Z_2$-equivariant homeomorphism
$h:(K*[3])^{*2}_\Delta \to \Sigma^2 K^{*2}_\Delta$.

(f) The maps
$$\pi^{d-1}_{\Z_2}(K^{\times2}_\Delta) \overset{q^*\Sigma}\to \pi^d_{\Z_2}(K^{*2}_\Delta) \overset{h^*\Sigma^2}\to \pi^{d+2}_{\Z_2}((K*[3])^{*2}_\Delta)$$
are 1--1 correspondences for $d\ge k+2$ and are surjective for $d\ge k+1$.
\end{lemma}

\begin{proof}
Of this lemma parts (a,b,c) are known, parts (d,e) are implicitly known, and part (f) easily follows from (c,d,e)
(see details and references below).

Part (a) follows because if $f:K\to\R^d$ is an embedding, then an equivariant map $\t f:K^{\times2}_\Delta\to S^{d-1}$ is defined by $\t f(x,y):=\dfrac{fx-fy}{|fx-fy|}$.

Part (b) is the non-trivial main result of \cite{We67}, see also the survey
\cite[\S5]{Sk06}.

Part (c) is  \cite[Theorem 2.5]{CF60}, see also \cite[Theorem 2.5]{Sk02}.

Part (d) is proved by setting $q[(x,y,t)]:=[(x,y,2t-1)]$.

Part (d) is also the conjunction of \cite[Exercise 4 to \S5.5]{Ma03} and a part of \cite[Cone Lemma 4.2.1]{Sk02}
(this part of \cite[Cone Lemma 4.2]{Sk02} is easy, \aronly{could have been known in folklore before \cite{Sk02},} and is essentially reproved in \cite[Lemma 6.1]{Me22}).

Part (e) proved by taking $\Z_2$-equivariant homeomorphisms
$$(K*[3])^{*2}_\Delta \cong K^{*2}_\Delta*[3]^{*2}_\Delta \cong K^{*2}_\Delta*S^1 \cong \Sigma^2 K^{*2}_\Delta,\quad\text{of which}$$

$\bullet$ the first holds by  \cite[Lemma 5.5.2]{Ma03} and is given by
$\left[\left([(x,i,s)],[(y,j,t)],\lambda\right)\right]\mapsto\left[\left([(x,y,s)],[(i,j,t)],\lambda\right)\right]$;

$\bullet$ the second holds because $[3]^{*2}_\Delta$ is the graph $K_{3,3}=[3]*[3]$ without the edges $(i,i')$, and so is the cycle of length 6 \cite[\S5]{Ma03};

$\bullet$ the third is well-known.

In order to prove (f), by (c,e) it suffices to prove that
$q^*: \pi^d_{\Z_2}(\Sigma(K^{\times2}_\Delta)) \to \pi^d_{\Z_2}(K^{*2}_\Delta)$
is a 1--1 correspondence for $d\ge k+2$ and is surjective for $d\ge k+1$.
A $\Z_2$-equivariant map
$K^{*2}_\Delta\to S^d$ is {\it nice} if it maps the non-trivial preimages of $q$
to (the opposite) points of $S^d$.
For $d\ge k+1$ any map
$K\to S^d$ is null-homotopic, so any $\Z_2$-equivariant map
$K^{*2}_\Delta\to S^d$ is $\Z_2$-equivariantly homotopic to a nice map, so $q^*$ is surjective.
For $d\ge k+2$ any map
$\Sigma K\to S^d$ is null-homotopic, so
any $\Z_2$-equivariant homotopy
$K^{*2}_\Delta\times I\to S^d$ between nice maps is homotopic to a homotopy through nice maps, so $q^*$ is injective.
\end{proof}



\begin{proof}[Proof of Theorem \ref{t:eq}\aronly{.a}]
Apply Lemma \ref{l:steps}.a,d
for the complex $K*L$ embeddable into $\R^{d+q+1}$.
Since $(K*L)^{*2}_\Delta \cong K^{*2}_\Delta*L^{*2}_\Delta$ \cite[Lemma 5.5.2]{Ma03}, we obtain
a $\Z_2$-equivariant map $\alpha:K^{*2}_\Delta* L^{*2}_\Delta\to S^{d+q+1}$.
Then $\alpha\circ(\id*\varphi) :  K^{*2}_\Delta*S^q\rightarrow S^{d+q+1}$ is a $\Z_2$-equivariant map.
There is a $\Z_2$-equivariant homeomorphism $K^{*2}_\Delta* S^q\cong \Sigma^{q+1}K^{*2}_\Delta$.
Since $k<d$, by Lemma \ref{l:steps}.c the equivariant suspension
$$\Sigma^{q+1}:\pi_{\Z_2}^d(K^{*2}_\Delta) \to \pi_{\Z_2}^{d+q+1}(\Sigma^{q+1}K^{*2}_\Delta)$$
is surjective.
Hence $\pi_{\Z_2}^d(K^{*2}_\Delta)\neq \emptyset$.
Since $k<d$, by Lemma \ref{l:steps}.f it follows that $\pi_{\Z_2}^{d-1}(K^{\times 2}_\Delta)\neq \emptyset$.
So by Lemma \ref{l:steps}.b $K$ embeds into $\R^d$.
\end{proof}

\aronly{

\begin{proof}[Proof of Theorem \ref{t:eq}.b]
Since $f$ is level preserving, there is a map $f_0:K*L\to \R^{d+q}$ such that $f([x,y,t])=f_0([x,y,t])\times t$.
For any $(x,x')\in K^{\times 2}_\Delta$, $(y,y')\in L^{\times 2}_{\Delta}$ and $t\in[0,1]$ we have $f_0([x,x',t])\neq f_0([y,y',t])$.
Hence a $\Z_2$-equivariant map
$$\tilde f: K^{\times 2}_\Delta * L^{\times 2}_\Delta \rightarrow
S^{d+q-1}\quad\text{is defined by}\quad
\t f([(x,x'),(y,y'),t])=\frac{f_0([x,x',t])-f_0([y,y',t])}{|f_0([x,x',t])-f_0([y,y',t])|}.$$
Then $\t f\circ (\id*\psi):K^{\times 2}_{\Delta}*S^{q-1} \rightarrow S^{d+q-1}$ is a $\Z_2$-equivariant map.


There is a $\Z_2$-equivariant homeomorphism $K^{\times2}_\Delta* S^{q-1}\cong \Sigma^qK^{\times2}_\Delta$.
Since $k<d$, by the surjectivity part of Lemma \ref{l:steps}.c the equivariant suspension
$$\Sigma^q:\pi_{\Z_2}^{d-1}(K^{\times2}_\Delta) \to \pi_{\Z_2}^{d+q-1}(\Sigma^qK^{\times2}_\Delta)$$
is surjective.
Hence $\pi_{\Z_2}^{d-1}(K^{\times2}_\Delta)\ne\emptyset$.
So by Lemma \ref{l:steps}.b $K$ embeds into $\R^d$.
\end{proof}

}


\begin{remark}\label{r:histjo} (a) Our
\aronly{proofs of Theorems \ref{t:parsaj} and \ref{t:eq}.b are simpler than those from
\cite[proofs of $(iv)\Rightarrow(i)$ of Corollary 4.4 and of Theorem 4.5]{MS06},
\cite[Proof of Theorem C in \S5 and \S6]{Me22}}
\jonly{proof of Theorem \ref{t:parsaj} is simpler than that from
\cite[proof of $(iv)\Rightarrow(i)$ of Corollary 4.4]{MS06}, \cite[Proof of Theorem C in \S5 and \S6]{Me22}}
because we use equivariant maps instead of the obstruction $\Theta^d$ whose definition (even for polyhedra) requires several pages.
In particular, we use the Weber theorem (Lemma \ref{l:steps}.b) instead of its reformulation in terms of the obstruction $\Theta^d$
\cite[Theorem 6.3]{Me06}.
We explicitly use the deleted join which is more convenient for calculations than the deleted product.
Thus although our proofs are clearer and shorter, they are
not alternative proofs based on very different ideas.

(b) A particular case of \cite[Lemma 9]{BKK} for $K_2$ being the three-point set is a homological mod 2 version of the case $d=2k$ of Theorem \ref{t:parsaj} (see also discussion before Proposition 5 in \cite{BKK} of  the condition (3) from Definition 4 of \cite{BKK}).

(c) The paper \cite{Pa20} proves the case $d=2k$ of Theorem \ref{t:parsaj} by proving that changing $K$ to $K*[3]$ raises by 2 the so called {\it Smith index} of the deleted product.

\aronly{(d) Theorem \ref{t:eq}.a in a sense generalizes the following Grunbaum-van Kampen-Flores theorem \cite{Gr69}: {\it if $K_i$ is the $k_i$-skeleton of the $(2k_i+2)$-simplex, $i=1,\ldots,p$, then the join $P=K_1*\ldots*K_p$ does not embed into the Euclidean space $\R^{2\dim P}$, where $\dim P=\sum_i d_i +p-1$.}
The latter result follows from Theorem \ref{t:eq}.a for $k_1,\ldots,k_p>2$ by induction on $p$ because the deleted join $K_{n,\Delta}^{*2}\cong_{\Z_2} S^{2k_n+1}$ \cite[page 117]{Ma03}, \cite{Gr69}.
However, the original proof is much simpler.
Namely, analogously by induction on $p$ the deleted join of $P$ is a sphere and the non-embeddability follows from the Borsuk-Ulam theorem.

(e) The existence of a $\Z_2$-equivariant map $S^q\rightarrow L^{*2}_{\Delta}$ implies the existence of a $\Z_2$-equivariant map $S^{q-1}\rightarrow L^{\times2}_{\Delta}$.
This follows by Lemmas \ref{l:steps}.c,d and \cite[Exercise 4 to \S5.5]{Ma03}, cf. proof of Lemma \ref{l:steps}.f.}
\end{remark}

{\it Books, surveys and expository papers in this list are marked by the stars.}


\begin{thebibliography}{RSS95}

\UseRawInputEncoding

\newcommand{\abc}{\bibitem[ABC+]{ABC+} * \emph{M. Atiyah, A. Borel, G. J. Chaitin, D. Friedan, J. Glimm, J. J. Gray, M. W. Hirsch, S. MacLane, B. B. Mandelbrot, D. Ruelle, A. Schwarz, K. Uhlenbeck, R. Thom, E. Witten, C.  Zeeman.} Responses to ``Theoretical Mathematics: Toward a cultural synthesis of mathematics and theoretical physics'', by A. Jaffe and F. Quinn. Bull. Am. Math. Soc. 30 (1994) 178--207. arXiv:math/9404229.}

\newcommand{\adnt}{\bibitem[Ad93]{Ad93} * \textit{M. Adachi}. Embeddings and Immersions. Amer. Math.
Soc., 1993. (Transl. of Math. Monographs; V.~124).}

\newcommand{\agles}{\bibitem[AGL]{AGL86} Mathematical Economics,  ed. by A. Ambrosetti, F. Gori, R. Lucchetti,
Lect. Notes Math. 1330, Springer, 1986.}


\newcommand{\akzz}{\bibitem[Ak00]{Ak00} * \emph{П. М. Ахметьев.} Вложения компактов, стабильные
гомотопические группы сфер и теория особенностей, Успехи Мат. Наук.  2000. 55:3. C.~3-62.}

\newcommand{\akoe}{\bibitem[AK19]{AK19} \emph{S. Avvakumov, R. Karasev.} Envy-free division using mapping degree,
Mathematika, 67:1 (2020), 36--53. arXiv:1907.11183.}

\newcommand{\akto}{\bibitem[AK21]{AK21} \emph{G. Arone and V. Krushkal.}
Embedding obstructions in $\R^d$ from the Goodwillie-Weiss calculus and Whitney disks. arXiv:2101.10995. }

\newcommand{\akm}{\bibitem[AKM]{AKM} \emph{M. Abrahamsen, L. Kleist and T. Miltzow.}
Geometric Embeddability of Complexes is $\exists\mathbb R$-complete, arXiv:2108.02585.}

\newcommand{\aksoe}{\bibitem[AKS]{AKS} \emph{S. Avvakumov, R. Karasev and A. Skopenkov.} Stronger counterexamples to the topological Tverberg conjecture, submitted. arXiv:1908.08731.}

\newcommand{\akuoe}{\bibitem[AKu19]{AKu19} \emph{S. Avvakumov, S. Kudrya.}
Vanishing of all equivariant obstructions and the mapping degree.
Discr. Comp. Geom., 66:3 (2021) 1202--1216. arXiv:1910.12628.}

\newcommand{\alto}{\bibitem[Al22]{Al22} \emph{E. Alkin,}
Hardness of almost embedding simplicial complexes in $\R^d$, II. arXiv:2206.13486}

\newcommand{\amsw}{\bibitem[AMS+]{AMSW} \emph{S. Avvakumov, I. Mabillard, A. Skopenkov and U. Wagner.}
Eliminating Higher-Multiplicity Intersections, III. Codimension 2, Israel J. Math. 245 (2021) 501--534.  arxiv:1511.03501.}


\newcommand{\anzt}{\bibitem[An03]{An03} * \emph{Д. В. Аносов.} Отображения окружности, векторные поля и их применения. М: МЦНМО, 2003.}

\newcommand{\arnf}{\bibitem[Ar95]{Ar95} \emph{V. I. Arnold,}  Topological invariants of plane curves and caustics, University Lecture Series, Vol. 5, Amer. Math. Soc., Providence, RI, 1995.}

\newcommand{\arszo}{\bibitem[ARS01]{ARS01} \emph{P. Akhmetiev, D. Repov\v s and A. Skopenkov},
Embedding products of low-dimensional manifolds in $\R^m$, Topol. Appl. 113 (2001), 7--12.}

\newcommand{\arszt}{\bibitem[ARS02]{ARS02} \emph{P. Akhmetiev, D. Repovs and A. Skopenkov.} Obstructions to approximating maps of $n$-manifolds into $R^{2n}$ by embeddings, Topol. Appl., 123 (2002), 3--14.}

\newcommand{\asoed}{\bibitem[As]{As} \emph{A. Asanau,} \lowercase{A SIMPLE PROOF THAT CONNECTED SUM OF ORDERED
ORIENTED LINKS IS NOT WELL-DEFINED,} Math. Notes, to appear.}

\newcommand{\asoe}{\bibitem[As]{As} \emph{A. Asanau,} On the \lowercase{TRIPLE SELF-INTERSECTION NUMBER FOR GRAPHS IN THE PLANE,} unpublished, 2018.}

\newcommand{\avos}{\bibitem[Av14]{Av14} \emph{S. Avvakumov,} The classification of certain linked 3-manifolds in 6-space, Moscow Math. J., 16:1 (2016), 1--25. arXiv:1408.3918.}

\newcommand{\avose}{\bibitem[Av17]{Av17} \emph{S. Avvakumov,} The classification of linked 3-manifolds in 6-space, Algebraic \& Geometric Topology, 22:6 (2022) 2587--2630. arXiv:1704.06501.}


\newcommand{\bant}{\bibitem[Ba93]{Ba93} * \emph{T. Bartsch.} Topological methods for variational problems with
symmetries, Lecture Notes in Mathematics, 1560, Springer-Verlag, Berlin, 1993.}

\newcommand{\bbsn}{\bibitem[BB79]{BB} \emph{E.~G. Bajm{{\'o}}czy and I.~B{{\'a}}r{{\'a}}ny,}
\newblock On a common generalization of {B}orsuk's and {R}adon's theorem,
\newblock Acta Math.\ Acad.\ Sci.\ Hungar.\ 34:3 (1979), 347-350.}

\newcommand{\bbzos}{\bibitem[BBZ]{BBZ} * \emph{I.~B{{\'a}}r{{\'a}}ny, P.~V.~M. Blagojevi{{\'c}} and G.~M. Ziegler.} Tverberg's Theorem at 50: Extensions and Counterexamples, Notices of the Amer. Math. Soc., 63:7 (2016), 732--739.}


\newcommand{\bcm}{\bibitem[BCM]{BCM} * 13th Hilbert Problem on superpositions of functions, presented by A. Belov, A. Chilikov, I. Mitrofanov, S. Shaposhnikov and A. Skopenkov,
\url{http://www.turgor.ru/lktg/2016/5/index.htm}.}

\newcommand{\beet}{\bibitem[BE82]{BE82} * \emph{V.G. Boltyansky and V.A. Efremovich.} Intuitive Combinatorial Topology. Springer.}

\newcommand{\beetr}{\bibitem[BE82]{BE82} * \emph{В. Г. Болтянский и В. А. Ефремович.} Наглядная топология. М.:  Наука, 1982.}


\newcommand{\bfzof}{\bibitem[BFZ14]{BFZ14} \emph{P. V. M. Blagojevi{\'c}, F. Frick, and G. M. Ziegler,}
Tverberg plus constraints, Bull. Lond. Math. Soc. 46:5 (2014), 953-967, arXiv:1401.0690.}


\newcommand{\bfzos}{\bibitem[BFZ]{BFZ} \emph{P. V. M. Blagojevi{\'c}, F. Frick and G. M. Ziegler,}
Barycenters of Polytope Skeleta and Counterexamples to the Topological Tverberg Conjecture, via Constraints,
J. Eur. Math. Soc., 21:7 (2019) 2107-2116. arXiv:1510.07984.}


\newcommand{\bgos}{\bibitem[BG16]{BG16} \emph{A. Bj\"orner and A. Goodarzi}, On Codimension one Embedding of Simplicial Complexes, in book: A Journey Through Discrete Mathematics, arXiv:1605.01240.}

\newcommand{\biet}{\bibitem[Bi83]{Bi83} * \emph{R. H. Bing.} The Geometric Topology of 3-Manifolds. Providence, R.~I. 1983. (Amer. Math. Soc. Colloq. Publ., 40).}

\newcommand{\bitz}{\bibitem[Bi20]{Bi20} * \emph{A. Bikeev.} Realizability of discs with ribbons on the M\"obius strip. Mat. Prosveschenie, 28 (2021), 150-158;
erratum to appear. arXiv:2010.15833.}

\newcommand{\bitzr}{\bibitem[Bi20]{Bi20} * \emph{А. Бикеев.} Реализуемость дисков с ленточками на ленте Мебиуса.
Мат. просвещение. Сер. 3. 28 (2021), 150--158.}

\newcommand{\bito}{\bibitem[Bi21]{Bi21} {\it A. I. Bikeev,}
Criteria for integer and modulo 2 embeddability of graphs to surfaces, arXiv:2012.12070v2.}


\newcommand{\bagos}{\bibitem[BG17]{BG17} \emph{S. Basu and S. Ghosh.} Equivariant maps related to the topological Tverberg conjecture, Homology, Homotopy and Applications 19:1 (2017) 155--170.}

\newcommand{\bkkmzof}{\bibitem[BKK]{BKK} \emph{M. Bestvina, M. Kapovich and B. Kleiner,}
Van Kampen's embedding obstruction for discrete groups, Invent. Math. 150 (2002) 219--235. arXiv:math/0010141.}

\newcommand{\bmzf}{\bibitem[BM04]{BM04} \emph{Boyer, J. M. and Myrvold, W. J.} On the cutting edge: simplified $O(n)$ planarity by edge addition,  Journal of Graph Algorithms and Applications, 8:3 (2004) 241--273.}

\newcommand{\bm}{\bibitem[BM15]{BM15} \emph{I. Bogdanov and A. Matushkin.} Algebraic proofs of linear versions of the Conway--Gordon--Sachs theorem and the van Kampen--Flores theorem, arXiv:1508.03185.}


\newcommand{\bmzzn}{\bibitem[BMZ09]{BMZ09} \emph{P. V. M. Blagojevi{\'c}, B. Matschke, G. M. Ziegler,}
Optimal bounds for a colorful Tverberg-Vre\'cica type problem, Advances in Math., 226 (2011), 5198-5215, arXiv:0911.2692.}

\newcommand{\bmzof}{\bibitem[BMZ15]{BMZ15} \emph{P. V. M. Blagojevi{\'c}, B. Matschke, G. M. Ziegler,}
Optimal bounds for the colored Tverberg problem, J. Eur. Math. Soc.,  17:4 (2015) 739--754,
arXiv:0910.4987.}

\newcommand{\bpns}{\bibitem[BP97]{BP97} * \emph{R. Benedetti and C. Petronio.} Branched standard spines of 3-manifolds, Lecture Notes in Math. 1653, Springer-Verlag, Berlin-Heidelberg-New York, 1997.}

\newcommand{\brst}{\bibitem[Br72]{Br72} \emph{J. L. Bryant.} Approximating embeddings of polyhedra in codimension 3, Trans. Amer. Math. Soc., 170 (1972) 85--95.}

\newcommand{\brts}{\bibitem[Br68]{Br68} \emph{P. Bruegel,} 1568,
\url{https://en.wikipedia.org/wiki/The_Magpie_on_the_Gallows}.}


\newcommand{\bren}{\bibitem[Br82]{brown1982} * \emph{K.~S. Brown.} \newblock Cohomology of Groups. \newblock Springer-Verlag New York, 1982.}


\newcommand{\bssos}{\bibitem[BS17]{BS17} * \emph{I.~B\'{a}r\'{a}ny and P. Sober\'{o}n,} Tverberg's theorem is 50 years old: a survey, Bull. Amer. Math. Soc. (N.S.) 55:4 (2018), 459--492. arXiv:1712.06119.}

\newcommand{\bsto}{\bibitem[BS21]{BS21} * \emph{A. Buchaev and A. Skopenkov,} Simple proofs of estimations of Ramsey numbers and of discrepancy, Mat. Prosveschenie, to appear, arXiv:2107.13831.}

\newcommand{\brsnn}{\bibitem[BRS99]{BRS99} \emph{D. Repov\v s, N. Brodsky and A. B. Skopenkov.}
A classification of 3-thickenings of 2-polyhedra, Topol. Appl. 1999. 94. P.~307-314.}

\newcommand{\bsseo}{\bibitem[BSS]{BSS} \emph{I.~B\'{a}r\'{a}ny, S.~B. Shlosman, and A.~Sz{\H{u}}cs,}
\newblock On a topological generalization of a theorem of {T}verberg,
\newblock J.\ London Math.\ Soc.\ (II. Ser.) 23 (1981), 158--164.}

\newcommand{\btzs}{\bibitem[BT07]{BT07} \emph{A. Bj\"orner, M. Tancer}, Combinatorial Alexander Duality --- a Short and Elementary Proof, Discr. and Comp. Geom., 42 (2009) 586. arXiv:0710.1172.}

\newcommand{\buse}{\bibitem[Bu68]{Bu68} \emph{A. R. Butz,} Space filling curves and mathematical programming, Information and Control, 12:4 (1968) 314--330.}


\newcommand{\bz}{\bibitem[BZ16]{BZ16} * \emph{P. V. M. Blagojevi\'c and G. M. Ziegler,} Beyond the Borsuk-Ulam theorem: The topological Tverberg story, in: A Journey Through Discrete Mathematics, Eds. M. Loebl,
J. Ne\v set\v ril, R. Thomas, Springer, 2017, 273--341. arXiv:1605.07321v3.}



\newcommand{\cano}{\bibitem[Ca91]{Ca91} * \emph{D. de Caen}, The ranks of tournament matrices, Amer. Math. Monthly, 98:9 (1991) 829--831.}

\newcommand{\ca}{\bibitem[Ca]{Ca} \emph{J. Carmesin.} Embedding simply connected 2-complexes in 3-space, I-V, arXiv:1709.04642, arXiv:1709.04643, arXiv:1709.04645, arXiv:1709.04652, arXiv:1709.04659.}

\newcommand{\cfsz}{\bibitem[CF60]{CF60} \emph{P. E. Conner and E. E. Floyd}, Fixed points free involutions and equivariant maps, Bull. Amer. Math. Soc., 66 (1960) 416--441.}

\newcommand{\cget}{\bibitem[CG83]{CG83} \emph{J. H. Conway and C. M. A. Gordon},
Knots and links in spatial graphs, J. Graph Theory  7 (1983), 445--453.}

\newcommand{\cten}{\bibitem[Ch]{Ch} \emph{Chuang Tzu,} translated by H. A. Giles, Bernard Quaritch, London, 1889.}

\newcommand{\ctruku}{\bibitem[Ch]{Ch} \emph{Chuang Tzu,} translated to Russian by S. Kuchera, in: Ancient Chinese Philosophy, v. I, Mysl, Moscow, 1972.}


\newcommand{\chnn}{\bibitem[Ch99]{Ch99} * \emph{А. В. Чернавский,} Теорема Жордана.  Мат. Просвещение, 3 (1999), 142--157.}

\newcommand{\hcon}{\bibitem[HC19]{HC19} * \emph{C. Herbert Clemens.} Two-Dimensional Geometries. A Problem-Solving Approach, Amer. Math. Soc., 2019.}

\newcommand{\ckmoo}{\bibitem[CKMS]{CKMS} \emph{M. \v Cadek, M. Kr\v c\'al. J. Matou\v sek, F. Sergeraert,
L. Vok\v r\'inek, U. Wagner.} Computing all maps into a sphere, J. of the ACM, 61:3 (2014). arXiv:1105.6257.}


\newcommand{\ckmvwot}{\bibitem[CKM12+]{CKM12+} \emph{M. \v Cadek, M. Kr\v c\'al. J. Matou\v sek, L. Vok\v r\'inek, U. Wagner.} Polynomial-time computation of homotopy groups and Postnikov systems in fixed dimension, SIAM J. Comput., 43:5 (2014), 1728--1780. arXiv:1211.3093.}

\newcommand{\ckmvw}{\bibitem[CKM+]{CKM+} \emph{M. \v Cadek, M. Kr\v c\'al. J. Matou\v sek, L. Vok\v r\'inek, U. Wagner.} Extendability of continuous maps is undecidable, Discr. and Comp. Geom. 51 (2014) 24--66.
arXiv:1302.2370.}

\newcommand{\ckppt}{\bibitem[CKP+]{CKP+} \emph{E. Colin de Verdi\'ere, V. Kalu\v za, P. Pat\'ak, Z. Pat\'akov\'a and M. Tancer.} A direct proof of the strong Hanani-Tutte theorem on the projective plane. Journal of Graph Algorithms and Applications, 21:5 (2017) 939--981.}

\newcommand{\cksof}{\bibitem[CKS+]{CKS+} * New ways of weaving baskets, presented by G. Chelnokov, Yu. Kudryashov, A.Skopenkov and A. Sossinsky, \url{http://www.turgor.ru/lktg/2004/lines.en/index.htm}.}

\newcommand{\ckv}{\bibitem[CKV]{CKV} \emph{M.~{\v{C}}adek, M.~Kr\v{c}\'{a}l, and L.~Vok\v{r}\'{\i}nek.}
Algorithmic solvability of the lifting-extension problem, Discr. Comp. Geom. 57 (2017), 915--965. arXiv:1307.6444.}


\newcommand{\clr}{\bibitem[CLR]{CLR} * \emph{Т. Кормен, Ч. Лейзерсон, Р. Ривест.} Алгоритмы:
построение и анализ, МЦНМО, Москва, 1999.}

\newcommand{\clreng}{\bibitem[CLR]{CLR} * \emph{T. H. Cormen, C. E.Leiserson, R. L.Rivest, C. Stein.} Introduction to Algorithms, MIT Press, 2009.}

\newcommand{\crzfru}{\bibitem[CR]{CR} * \emph{Р. Курант, Дж. Роббинс,} Что такое математика. М.: МЦНМО, 2004.}

\newcommand{\crzfen}{\bibitem[CR]{CR} * \emph{R. Courant and H. Robbins,} What is Mathematics, Oxford Univ. Press.}

\newcommand{\crsne}{\bibitem[CRS98]{CRS98} * \emph{A. Cavicchioli, D. Repov\v s and A. B. Skopenkov.}
Open problems on graphs, arising from geometric topology, Topol. Appl. 1998. 84. P.~207-226.}

\newcommand{\crsot}{\bibitem[CRS]{CRS} \emph{M. Cencelj, D. Repov\v s and M. Skopenkov,}
Classification of knotted tori in the 2-metastable dimension, Mat. Sbornik, 203:11 (2012), 1654--1681.
arxiv:math/0811.2745.}

\newcommand{\csoo}{\bibitem[CS08]{CS08} \emph{D. Crowley and A. Skopenkov.} A classification of smooth embeddings of 4-manifolds in 7-space, II, Intern. J. Math., 22:6 (2011) 731-757, arxiv:math/0808.1795.}

\newcommand{\csos}{\bibitem[CS16]{CS16} \emph{D. Crowley and A. Skopenkov,} Embeddings of non-simply-connected 4-manifolds in 7-space. I. Classification modulo knots, Moscow Math. J., 21 (2021), 43--98. arXiv:1611.04738.}


\newcommand{\csoso}{\bibitem[CS16o]{CS16o} \emph{D. Crowley and A. Skopenkov,} Embeddings of non-simply-connected 4-manifolds in 7-space. II. On the smooth classification, Proc. A of the Royal Soc. of Edinburgh 152:1 (2022), 163--181. arXiv:1612.04776.}


\newcommand{\crsk}{\bibitem[CS]{CS} \emph{D. Crowley and A. Skopenkov,} Embeddings of non-simply-connected 4-manifolds in 7-space. III. Piecewise-linear classification. draft.}

\newcommand{\cutz}{\bibitem[Cu20]{Cu20} \emph{C. Culter,} Cantor sets are not tangent homogeneous,
Topol. Appl. 271 (2020) 1--9.}


\newcommand{\dies}{\bibitem[Di87]{Di} * \emph{T. tom Dieck,} Transformation groups, Studies in Mathematics, vol. 8, Walter de Gruyter, Berlin, 1987.}

\newcommand{\dent}{\bibitem[De93]{De93}  \emph{T.K. Dey.} On counting triangulations in $d$-dimensions. Comput. Geom.  3:6 (1993) 315--325.}

\newcommand{\denf}{\bibitem[DE94]{DE94}  \emph{T.K. Dey and H. Edelsbrunner.} Counting triangle crossings and halving planes, Discrete Comput. Geom. 12 (1994), 281--289.}

\newcommand{\dgn}{\bibitem[DGN+]{DGN+} * Low rank matrix completion, presented by S. Dzhenzher, T. Garaev, O. Nikitenko, A. Petukhov, A. Skopenkov, A. Voropaev,
\url{https://www.mccme.ru/circles/oim/netflix.pdf} }

\newcommand{\dstt}{\bibitem[DS22]{DS22}  \emph{S. Dzhenzher and A. Skopenkov,} To the K\"uhnel conjecture on embeddability of $k$-complexes into $2k$-manifolds, arXiv:2208.04188.}

\newcommand{\embo}{\bibitem[Eb]{Eb} * \url{http://www.map.mpim-bonn.mpg.de/Embeddings_of_manifolds_with_boundary:_classification}}

\newcommand{\embe}{\bibitem[Em]{Em} * \url{http://www.map.mpim-bonn.mpg.de/Embedding_(simple_definition)}}

\newcommand{\ers}{\bibitem[ERS]{ERS} * Invariants of graph drawings in the plane, presented by A. Enne, A. Ryabichev, A. Skopenkov and T. Zaitsev, \url{http://www.turgor.ru/lktg/2017/6/index.htm}}


\newcommand{\feto}{\bibitem[Fe21]{Fe21} \emph{M. Fedorov.} A description of values of Seifert form for punctured $n$-manifolds in $(2n-1)$-space, arXiv:2107.02541.}

\newcommand{\ffen}{\bibitem[FF89]{FF89} * \emph{А. Т. Фоменко и Д. Б. Фукс.} Курс гомотопической топологии. М.: Наука, 1989.}

\newcommand{\ffene}{\bibitem[FF89]{FF89} * \emph{A.T. Fomenko and D.B. Fuchs.} Homotopical Topology, Springer, 2016.}


\newcommand{\fhzo}{\bibitem[FH10]{FH10}  \emph{M. Farber, E. Hanbury}. Topology of Configuration Space of Two Particles on a Graph, II. Algebr. Geom. Topol. 10 (2010) 2203--2227. arXiv:1005.2300.}


\newcommand{\fkosc}{\bibitem[FK17]{FK17} \emph{R. Fulek, J. Kyn{\v{c}}l,} Counterexample to an Extension of the Hanani-Tutte Theorem on the Surface of Genus 4, Combinatorica, 39 (2019) 1267--1279, arXiv:1709.00508.}

\newcommand{\fkos}{\bibitem[FK17]{FK17} \emph{R. Fulek, J. Kyn{\v{c}}l,} Hanani-Tutte for approximating maps of graphs, arXiv:1705.05243.}

\newcommand{\fkon}{\bibitem[FK19]{FK19} \emph{R. Fulek, J. Kyn{\v{c}}l,}
$\Z_2$-genus of graphs and minimum rank of partial symmetric matrices,
35th Intern. Symp. on Comp. Geom. (SoCG 2019), Article No. 39; pp. 39:1--39:16,
\url{https://drops.dagstuhl.de/opus/volltexte/2019/10443/pdf/LIPIcs-SoCG-2019-39.pdf}.
We refer to numbering in arXiv version: arXiv:1903.08637.}

\newcommand{\fktnf}{\bibitem[FKT]{FKT} \emph{M. H. Freedman, V. S. Krushkal and P. Teichner.} Van Kampen's
embedding obstruction is incomplete for 2-complexes in~$\R^4$, Math. Res. Letters. 1994. 1. P.~167-176.}

\newcommand{\fltf}{\bibitem[Fl34]{Fl34} \emph{A. Flores}, \"Uber $n$-dimensionale Komplexe die im $E^{2n+1}$ absolut selbstverschlungen sind, Ergeb. Math. Koll. 6 (1934) 4--7.}

\newcommand{\fo}{\bibitem[Fo]{Fo} * \emph{L. Fortnow.} Time for Computer Science to Grow Up,  \url{https://people.cs.uchicago.edu/~fortnow/papers/growup.pdf}.}

\newcommand{\fozf}{\bibitem[Fo04]{Fo04} * \emph{R. Fokkink.} A forgotten mathematician, Eur. Math. Soc. Newsletter 52 (2004) 9--14.}


\newcommand{\fpstz}{\bibitem[FPS]{FPS} \emph{R. Fulek, M.J. Pelsmajer and M. Schaefer.}
Strong Hanani-Tutte for the Torus, arXiv:2009.01683.}

\newcommand{\frse}{\bibitem[Fr78]{Fr78} \emph{M. Freedman,} Quadruple points of 3-manifolds in $S^4$, Comment. Math. Helv. 53 (1978), 385-394.}

\newcommand{\fres}{\bibitem[FR86]{FR86} \emph{R. Fenn, D. Rolfsen.}
Spheres may link homotopically in 4-space, J. London Math. Soc. 34 (1986) 177-184.}

\newcommand{\frofea}{\bibitem[Fr15']{Fr15'} \emph{F. Frick}, Counterexamples to the topological Tverberg conjecture, arXiv:1502.00947v1.}


\newcommand{\frof}{\bibitem[Fr15]{Fr15} \emph{F. Frick}, Counterexamples to the topological Tverberg conjecture,
Oberwolfach reports, 12:1 (2015), 318--321. arXiv:1502.00947.}

\newcommand{\fros}{\bibitem[Fr17]{Fr17} \emph{F. Frick}, O\lowercase{N AFFINE TVERBERG-TYPE RESULTS WITHOUT CONTINUOUS GENERALIZATION}, arXiv:1702.05466}


\newcommand{\fstz}{\bibitem[FS20]{FS20} \emph{F. Frick and P. Sober\'on}, The topological Tverberg problem beyond prime powers, arXiv:2005.05251.}

\newcommand{\fvto}{\bibitem[FV21]{FV21} \emph{M. Filakovsk\'y, L. Vok\v r\'inek.} Computing homotopy classes for diagrams, 	arXiv:2104.10152.}

\newcommand{\fwz}{\bibitem[FWZ]{FWZ} \emph{M. Filakovsk\'y, U. Wagner, S. Zhechev.} Embeddability of simplicial complexes is undecidable. Oberwolfach reports, to appear.}

\newcommand{\fwztz}{\bibitem[FWZ]{FWZ} \emph{M. Filakovsk\'y, U. Wagner, S. Zhechev.} Embeddability of simplicial complexes is undecidable. Proceedings of the 2020 ACM-SIAM Symposium on Discrete Algorithms.}



\newcommand{\ga}{\bibitem[GA]{GA} * \url{https://en.wikipedia.org/wiki/Galactic_algorithm}}

\newcommand{\gatt}{\bibitem[Ga22]{Ga22} T. Garaev, On winding numbers in $K_5$ minus an edge drawn in the plane, draft.}

\newcommand{\gdikrse}{\bibitem[GDI]{GDI} * {\it A. Chernov, A. Daynyak, A. Glibichuk, M. Ilyinskiy, A. Kupavskiy, A. Raigorodskiy and A. Skopenkov,} Elements of Discrete Mathematics As a Sequence of Problems (in Russian),
MCCME, Moscow, 2016. Update: \url{http://www.mccme.ru/circles/oim/discrbook.pdf}}

\newcommand{\gdikrs}{\bibitem[GDI]{GDI} * {\it А.А. Глибичук, А.Б. Дайняк, Д.Г. Ильинский, А.Б. Купавский, А.М. Райгородский, А.Б. Скопенков, А.А. Чернов,} Элементы дискретной математики в задачах, М, МЦНМО, 2016.
\url{http://www.mccme.ru/circles/oim/discrbook.pdf}}

\newcommand{\giso}{\bibitem[Gi71]{Gi71} * {\it S. Gitler,} Immersion and Embedding of Manifolds,
Proc. Symp. Pure Math. 22, 87-96 (1971).}

\newcommand{\gkp}{\bibitem[GKP]{GKP} * {\it R. Graham, D. Knuth, and O. Patashnik,} Concrete Mathematics: A Foundation for Computer Science, Addison–Wesley, first published in 1989, \url{https://www.csie.ntu.edu.tw/~r97002/temp/Concrete\%20Mathematics\%202e.pdf}.}

\newcommand{\gmpptw}{\bibitem[GMP+]{GMP+} \emph{X. Goaoc, I. Mabillard, P. Pat\'ak, Z. Pat\'akov\'a, M. Tancer, U. Wagner}, On Generalized Heawood Inequalities for Manifolds: a van Kampen--Flores-type Nonembeddability Result,
Israel J. Math., 222(2) (2017) 841-866. arXiv:1610.09063.}

\newcommand{\group}{\bibitem[Gr]{Gr} * \url{https://en.wikipedia.org/wiki/Groupthink}}

\newcommand{\grsz}{\bibitem[Gr69]{Gr69} \emph{B. Gr\"unbaum.} Imbeddings of simplicial complexes. Comment. Math. Helv., 44:1, 502--513, 1969.}


\newcommand{\gres}{\bibitem[Gr86]{Gr86} * \emph{M. Gromov}, Partial Differential Relations,
Ergebnisse der Mathematik und ihrer Grenzgebiete (3), Springer Verlag, Berlin-New York, 1986.}

\newcommand{\groz}{\bibitem[Gr10]{Gr10} \emph{M. Gromov,}
\newblock Singularities, expanders and topology of maps. Part 2: From combinatorics to topology via algebraic isoperimetry, \newblock Geometric and Functional Analysis 20 (2010), no.~2, 416--526.}

\newcommand{\grsn}{\bibitem[GR79]{GR79} \emph{J. L. Gross	and R. H. Rosen}, A linear time planarity algorithm for 2-complexes, Journal of the ACM, 26:4 (1979), 611--617.}

\newcommand{\gs}{\bibitem[GS]{GS} \emph{М. Гортинский и О. Скрябин.} Критерий вложимости графов в плоскость вдоль прямой, препринт.}

\newcommand{\gssn}{\bibitem[GS79]{GS} \emph{P.~M. Gruber and R.~Schneider,} Problems in geometric convexity. In {\em Contributions to geometry ({P}roc. {G}eom. {S}ympos., {S}iegen, 1978)}, 255--278. Birkh{\"a}user, Basel-Boston, Mass., 1979.}

\newcommand{\gsnn}{\bibitem[GS99]{GS99} \emph{R. Gompf and A. Stipsicz,}
4-manifolds and Kirby calculus, GSM20, AMS, Providence, RI, 1999.}


\newcommand{\gszs}{\bibitem[GS06]{GS06} \emph{D. Goncalves and A. Skopenkov,} Embeddings of homology equivalent manifolds with boundary, Topol. Appl., 153:12 (2006) 2026-2034. arxiv:1207.1326.}

\newcommand{\gssoe}{\bibitem[GSS+]{GSS+} * Projections of skew lines, presented by A. Gaifullin, A. Shapovalov, A. Skopenkov and M. Skopenkov, \url{http://www.turgor.ru/lktg/2001/index.php}.}

\newcommand{\gtes}{\bibitem[GT87]{GT87} * \emph{J. L. Gross and T. W. Tucker.}
Topological graph theory. New York: Wiley-Interscience, 1987.}

\newcommand{\guzn}{\bibitem[Gu09]{Gu09} \emph{A. Gundert.} On the complexity of embeddable simplicial complexes. Diplomarbeit, Freie Universit\"at Berlin, 2009. 	arXiv:1812.08447.}


\newcommand{\ha}{\bibitem[Ha]{Ha} * \emph{F. Harary.} Graph theory.
Рус. пер.: Ф. Харари. Теория графов. М., Мир, 1973.}

\newcommand{\hats}{\bibitem[Ha37]{Ha37} \emph{W. Hantzsche,} Einlagerung von Mannigfaltigkeiten in euklidische R\" aume, Math. Zeitschrift, 43:1 (1937) 38--58.}

\newcommand{\hastk}{\bibitem[Ha62k]{Ha62k} {\em A.~Haefliger,}  Knotted $(4k-1)$-spheres in $6k$-space, Ann. of Math. 75 (1962) 452--466.}

\newcommand{\hastl}{\bibitem[Ha62l]{Ha62l} \emph{A. Haefliger,} Differentiable links, Topology, 1 (1962) 241--244.}

\newcommand{\hast}{\bibitem[Ha63]{Ha63} \emph{A.~Haefliger,} Plongements differentiables dans le domain stable, Comment. Math. Helv. 36 (1962-63) 155--176.}

\newcommand{\hassa}{\bibitem[Ha66A]{Ha66A} \textit{A. Haefliger}. Differential embeddings of~$S^n$ in $S^{n+q}$ for $q>2$. Ann. Math. (2), 83 (1966), 402--~436.}

\newcommand{\hass}{\bibitem[Ha66C]{Ha66C} \emph{A.~Haefliger,}  Enlacements de spheres en codimension superiure a 2, Comment. Math. Helv. 41 (1966-67) 51--72.}

\newcommand{\hase}{\bibitem[Ha68]{Ha68} \emph{A. Haefliger,} Knotted Spheres and Related Geometric Topic,
in Proc. Int. Congr. Math., Moscow, 1966 (Mir, Moscow, 1968), 437--445.}

\newcommand{\hasn}{\bibitem[Ha69]{Ha69} \emph{L.~S.~Harris,} Intersections and embeddings of polyhedra, Topology 8 (1969) 1--26.}

\newcommand{\hasf}{\bibitem[Ha74]{Ha74} * \emph{P. Halmos,} How to talk mathematics. Notices of the Amer. Math. Soc., 21 (1974) 155--158.}

\newcommand{\haef}{\bibitem[Ha84]{Ha84} \emph{N. Habegger,} Obstruction to embedding disks II: a proof of a conjecture by Hudson, Topol. Appl. 17 (1984).}

\newcommand{\haes}{\bibitem[Ha86]{Ha86} \emph{N. Habegger,} Knots and links in codimension greater than 2, Topology, 25:3 (1986) 253--260.}

\newcommand{\hifn}{\bibitem[Hi59]{Hi59} \emph{M. W. Hirsch.} Immersions of manifolds, Trans. Amer. Math. Soc. 93 (1959) 242--276.}

\newcommand{\hjsf}{\bibitem[HJ64]{HJ64} \emph{R. Halin and H. A. Jung.}
Karakterisierung der Komplexe der Ebene und der 2-Sph\"are, Arch. Math. 1964. 15. P.~466-469.}

\newcommand{\hkne}{\bibitem[HK98]{HK98} \emph{N. Habegger and U. Kaiser,} Link homotopy in 2--metastable range, Topology 37:1 (1998) 75--94.}

\newcommand{\hmsnt}{\bibitem[HMS]{HMS93} * \emph{C. Hog-Angeloni, W. Metzler and A. J. Sieradski.}
Two-dimensional homotopy and combinatorial group theory. Cambridge: Cambridge Univ. Press, 1993. (London Math. Soc. Lecture Notes, 197).}

\newcommand{\ho}{\bibitem[Ho]{Ho} * The Hopf fibration, \url{https://www.youtube.com/watch?v=AKotMPGFJYk}}

\newcommand{\hozs}{\bibitem[Ho06]{Ho06} \emph{H. van der Holst,} Graphs and obstructions in four dimensions, J. Combin. Theory Ser. B 96:3 (2006), 388--404.}


\newcommand{\hpzn}{\bibitem[HP09]{HP09} \emph{H. van der Holst and R. Pendavingh,} On a graph property generalizing planarity and flatness, Combinatorica, 29 (2009) 337--361.}

\newcommand{\htsf}{\bibitem[HT74]{HT74} \emph{J. Hopcroft and R. E. Tarjan,} Efficient planarity testing, J. of the Association for Computing Machinery, 21:4 (1974) 549--568.}

\newcommand{\hufn}{\bibitem[Hu59]{hu59} * \emph{S. T. Hu,} Homotopy Theory, Academic Press, New York, 1959.}

\newcommand{\husn}{\bibitem[Hu69]{Hu69} * \emph{J. F. P. Hudson.} Piecewise linear topology, W. A. Benjamin, Inc., New York-Amsterdam, 1969.}


\newcommand{\io}{\bibitem[Io]{Io} * \url{https://en.wikipedia.org/wiki/Category:Impossible_objects}}

\newcommand{\info}{\bibitem[IF]{IF} * \url{http://www.map.mpim-bonn.mpg.de/Intersection_form}}

\newcommand{\irsf}{\bibitem[Ir65]{Ir65} \emph{M.~C.~Irwin,} Embeddings of polyhedral manifolds, Ann. of Math. (2)
82 (1965) 1--14.}

\newcommand{\isot}{\bibitem[Is]{Is} * \url{http://www.map.mpim-bonn.mpg.de/Isotopy}}


\newcommand{\jqnt}{\bibitem[JQ93]{JQ93} * \emph{A. Jaffe, F. Quinn,} ``Theoretical mathematics'': Toward a cultural synthesis of mathematics and theoretical physics. Bull.Am.Math.Soc. 29 (1993) 1-13. arXiv:math/9307227.}

\newcommand{\jozt}{\bibitem[Jo02]{Jo02} \emph{C. M. Johnson.} An obstruction to embedding a simplicial $n$-complex into a $2n$-manifold, Topology Appl. 122:3 (2002) 581--591.}

\newcommand{\jvz}{\bibitem[JVZ]{JVZ} D. Joji\'c, S. T. Vre\'cica, R. T. \v Zivaljevi\' c,
Topology and combinatorics of 'unavoidable complexes', arXiv:1603.08472v1.}


\newcommand{\kalai}{\bibitem[Ka]{Ka} G. Kalai, From Oberwolfach: The Topological Tverberg Conjecture is False, `Combinatorics and more' blog post, February 6, 2015, \url{gilkalai.wordpress.com}}

\newcommand{\kh}{\bibitem[Kh]{Kh} \emph{А.И. Храбров.} Руководство по чтению лекций
\url{http://vm.tstu.tver.ru/topics/pdf_tests/lection.pdf}}

\newcommand{\kho}{\bibitem[Kho]{Kho} \emph{N. Khoroshavkina.} A simple characterization of graphs of cutwidth 2, arXiv:1811.06716.}

\newcommand{\kkrot}{\bibitem[KKR]{KKR} \emph{K. Kawarabayashi, Y. Kobayashi and B. Reed.} The disjoint paths problem in quadratic time, J. of Comb. Theory, Ser. B, 102:2 (2012), 424--435.}

\newcommand{\kmsth}{\bibitem[KM63]{KM63} \emph{M. A. Kervaire and J. W. Milnor,} Groups of homotopy spheres. I,  Ann. of Math. (2) 77 (1963), 504-537.}

\newcommand{\kozeru}{\bibitem[Ko18]{Ko18} * \emph{Е. Колпаков.}
Доказательство теоремы Радона при помощи понижения размерности, Мат. Просвещение, 23 (2018), arXiv:1903.11055.}

\newcommand{\koze}{\bibitem[Ko18]{Ko18} * \emph{E. Kolpakov.}
A proof of Radon Theorem via lowering of dimension, Mat. Prosveschenie, 23 (2018), arXiv:1903.11055.}

\newcommand{\ko}{\bibitem[Ko]{Ko} \emph{E. Kolpakov.} A `converse' to the Constraint Lemma, arXiv:1903.08910.}

\newcommand{\koon}{\bibitem[Ko19]{Ko19} \emph{E. Kogan.} Linking of three triangles in 3-space, arXiv:1908.03865.}

\newcommand{\koto}{\bibitem[Ko21]{Ko21} \emph{E. Kogan.} On the rank of $\Z_2$-matrices with free entries on the diagonal, arXiv:2104.10668.}

\newcommand{\koee}{\bibitem[Ko88]{Ko88} \emph{U. Koschorke.} Link maps and the geometry of their invariants,
Manuscripta Math. 61:4 (1988) 383--415.}

\newcommand{\kps}{\bibitem[KPS]{KPS} * \emph{A. Kaibkhanov, D. Permyakov and A. Skopenkov.}
Realization of graphs with rotation, \url{http://www.turgor.ru/lktg/2005/3/index.htm}.}

\newcommand{\krzz}{\bibitem[Kr00]{Kr00} \emph{V. S. Krushkal.} Embedding obstructions and 4-dimensional thickenings of 2-complexes, Proc. Amer. Math. Soc. 128:12 (2000) 3683--3691. arXiv:math/0004058. }

\newcommand{\ksnn}{\bibitem[KS99]{KS99} * \emph{П. Кожевников и А. Скопенков.} Узкие деревья на плоскости, Мат. Образование. 1999. 2-3. С.~126-131.}

\newcommand{\kstz}{\bibitem[KS20]{KS20} \emph{R. Karasev and A. Skopenkov.}
Some `converses' to intrinsic linking theorems, Discr. Comp. Geom., to appear, arXiv:2008.02523.}

\newcommand{\ksto}{\bibitem[KS21]{KS21} * \emph{E. Kogan and A. Skopenkov.} A short exposition of the Patak-Tancer theorem on non-embeddability of $k$-complexes in $2k$-manifolds,  arXiv:2106.14010.}

\newcommand{\kstoe}{\bibitem[KS21e]{KS21e} \emph{E. Kogan and A. Skopenkov.}
Embeddings of $k$-complexes in $2k$-manifolds and minimum rank of partial symmetric matrices, arXiv:2112.06636.}

\newcommand{\kuse}{\bibitem[Ku68]{Ku68} * \emph{К. Куратовский.} Топология. Т.~1,~2. М.: Мир, 1969.}

\newcommand{\kunfo}{\bibitem[Ku94]{Ku94} \emph{W. K\"uhnel.} Manifolds in the skeletons of convex polytopes, tightness, and generalized Heawood inequalities. In Polytopes: abstract, convex and computational (Scarborough, ON, 1993), volume 440 of NATO Adv. Sci. Inst. Ser. C Math. Phys. Sci., pp. 241--247. Kluwer
Acad. Publ., Dordrecht, 1994.}


\newcommand{\kunf}{\bibitem[Ku95]{Ku95} * \emph{W. K\"uhnel}, Tight Polyhedral Submanifolds and Tight Triangulations, Lecture Notes in Math. 1612, Springer, 1995.}


\newcommand{\lazz}{\bibitem[La00]{La00} \emph{F. Lasheras.} An obstruction to 3-dimensional thickening,
Proc. Amer. Math. Soc. 2000. 128. P.~893-902.}

\newcommand{\lfma}{\bibitem[LF]{LF} \url{http://www.map.mpim-bonn.mpg.de/Linking_form}}

\newcommand{\lloe}{\bibitem[LL18]{LL18} \emph{A.S. Levine and T. Lidman.} Simply connected, spineless 4-manifolds, Forum of Math., Sigma, 7 (2019) e14, 1--11, arxiv:1803.01765.}

\newcommand{\lo}{\bibitem[Lo]{Lo} M.~de~Longueville. Notes on the topological Tverberg theorem.
Discrete Math.  247 (2002), no.~1--3, 271--297.
(The paper first appeared in
Discrete Math. 241 (2001) 207--233, but the original version suffered from serious publisher's typesetting errors.)}

\newcommand{\loot}{\bibitem[Lo13]{Lo13} \emph{M. de Longueville.} A course in topological combinatorics. Universitext. Springer, New York (2013).}

\newcommand{\lssn}{\bibitem[LS69]{LS69} \emph{W. B. R. Lickorish and L. C. Siebenmann.}
Regular neighborhoods and the stable range,  Trans. Amer. Math. Soc.. 1969. 139. P.~207-230.}

\newcommand{\lsne}{\bibitem[LS98]{LS98} \emph{L. Lovasz and A. Schrijver,}
A Borsuk theorem for antipodal links and a spectral characterization of linklessly embeddable graphs, Proc. Amer. Math. Soc. 126:5 (1998), 1275-1285.}

\newcommand{\ltof}{\bibitem[LT14]{LT14} \emph{E. Lindenstrauss and M. Tsukamoto,} Mean dimension and an embedding problem: an example, Israel J. Math. 199 (2014).}


\newcommand{\lyzf}{\bibitem[LY04]{LY04} * \emph{Y. Lin and A. Yang,} On 3-cutwidth critical graphs, Discrete Mathematics, 275 (2004), 339--346.}

\newcommand{\lz}{\bibitem[LZ]{LZ} * \emph{S. Lando and A. Zvonkin.} Embedded Graphs. Springer.}



\newcommand{\mast}{\bibitem[Ma73]{Ma73} \emph{С. В. Матвеев.} Специальные остовы кусочно-линейных многообразий, Мат. Сборник. 1973. 92. С.~282-293.}

\newcommand{\maste}{\bibitem[Ma73]{Ma73} \emph{S. V. Matveev.} Special skeletons of PL manifolds (in Russian), Mat. Sbornik. 1973. 92. P.~282-293.}

\newcommand{\mans}{\bibitem[Ma97]{Ma97} \emph{Yu. Makarychev.} A short proof of Kuratowski's graph planarity criterion, J. of Graph Theory, 25 (1997), 129--131.}

\newcommand{\mazt}{\bibitem[Ma03]{Ma03} * \emph{J.~Matou{\v{s}}ek.} Using the {B}orsuk-{U}lam theorem:
Lectures on topological methods in combinatorics and geometry. Springer Verlag, 2008.}


\newcommand{\mazf}{\bibitem[Ma05]{Ma05} \emph{V. Manturov.} A proof of the Vasiliev conjecture on the planarity of singular links, Izv. RAN 2005.}

\newcommand{\metn}{\bibitem[Me29]{Me29} \emph{K. Menger.} \"Uber pl\"attbare Dreiergraphen und Potenzen nicht pl\"attbarer Graphen, Ergebnisse Math. Kolloq., 2 (1929) 30--31.}

\newcommand{\mezf}{\bibitem[Me04]{Me04} \emph{S. Melikhov.} Sphere eversions and realization of mappings, Trudy MIAN 247 (2004) 159-181 (in Russian) arXiv:math.GT/0305158.}

\newcommand{\mezs}{\bibitem[Me06]{Me06} \emph{S. A. Melikhov}, The van Kampen obstruction and its relatives, 	
Proc. Steklov Inst. Math 266 (2009), 142-176 (= Trudy MIAN 266 (2009), 149-183), arXiv:math/0612082.}

\newcommand{\meoo}{\bibitem[Me11]{Me11} \emph{S. A. Melikhov}, Combinatorics of embeddings, arXiv:1103.5457.}

\newcommand{\meos}{\bibitem[Me17]{Me17} \emph{S. Melikhov,} Gauss type formulas for link map invariants, arXiv:1711.03530.}

\newcommand{\meoe}{\bibitem[Me18]{Me18} \emph{S. A. Melikhov,} A triple-point Whitney trick, J. Topol. Anal., 2018, 1--6. arXiv:2210.04016.}


\newcommand{\mett}{\bibitem[Me22]{Me22} \emph{S. A. Melikhov,} Embeddability of joins and products of polyhedra, Topol. Methods in Nonlinear Analysis, 60:1 (2022), 185-201. arXiv:2210.04015}

\newcommand{\miso}{\bibitem[Mi61]{Mi61} \emph{J. Milnor,} A procedure for killing homotopy groups of differentiable manifolds, Proc. Sympos. Pure Math, Vol. III (1961), 39--55.}

\newcommand{\mins}{\bibitem[Mi97]{Mi97} \emph{P. Minc.} Embedding simplicial arcs into the plane, Topol. Proc. 1997. 22. 305--340.}


\newcommand{\moss}{\bibitem[Mo77]{Mo77} * \emph{E. E. Moise.} Geometric Topology in Dimensions 2 and 3 (GTM), Springer-Verlag, 1977.}

\newcommand{\moen}{\bibitem[Mo89]{Mo89} \textit{B. Mohar}. An obstruction to embedding graphs in
surfaces. Discrete Math. 78 (1989) 135--142.}

\newcommand{\mrst}{\bibitem[MRS+]{MRS+} \emph{A. de Mesmay, Y. Rieck, E. Sedgwick, M. Tancer,}
Embeddability in $\R^3$ is NP-hard. arXiv:1708.07734.}

\newcommand{\mesczs}{\bibitem[MS06]{MS06} \emph{S.A. Melikhov, E.V. Shchepin,} The telescope approach to embeddability of compacta. arXiv:math.GT/0612085.}

\newcommand{\mstwof}{\bibitem[MST+]{MST+} \emph{J. Matou\v sek, E. Sedgwick, M. Tancer, U. Wagner}, Embeddability in the 3-sphere is decidable, Journal of the ACM 65:1 (2018) 1--49, arXiv:1402.0815.}


\newcommand{\mtzo}{\bibitem[MT01]{MT01} * \emph{B. Mohar and C. Thomassen.} Graphs on Surfaces.
The John Hopkins University Press, 2001.}

\newcommand{\mtwoz}{\bibitem[MTW10]{MTW10} \emph{J. Matou\v sek, M. Tancer, U. Wagner.} A geometric proof of
the colored Tverberg theorem, Discr. and Comp. Geometry, 47:2 (2012), 245--265. arXiv:1008.5275.}


\newcommand{\mtwoo}{\bibitem[MTW]{MTW} \emph{J. Matou\v sek, M. Tancer, U. Wagner.}
Hardness of embedding simplicial complexes in $\R^d$, J. Eur. Math. Soc. 13:2 (2011), 259--295. arXiv:0807.0336.}


\newcommand{\mwofo}{\bibitem[MW14]{MW14} \emph{I. Mabillard and U. Wagner.} Eliminating Tverberg Points, I. An Analogue of the Whitney Trick, Proc. of the 30th Annual Symp. on Comp. Geom. (SoCG'14), ACM, New York, 2014, pp. 171--180.}

\newcommand{\mwof}{\bibitem[MW15]{MW15} \emph{I. Mabillard and U. Wagner.}
Eliminating Higher-Multiplicity Intersections, I. A Whitney Trick for Tverberg-Type Problems. arXiv:1508.02349.}


\newcommand{\mwos}{\bibitem[MW16]{MW16} \emph{I. Mabillard and U. Wagner.} Eliminating Higher-Multiplicity Intersections, II. The Deleted Product Criterion in the $r$-Metastable Range. arXiv:1601.00876v2.}

\newcommand{\mwosd}{\bibitem[MW16']{MW16'} \emph{I. Mabillard and U. Wagner.} Eliminating Higher-Multiplicity Intersections, II. The Deleted Product Criterion in the r-Metastable Range,
Proceedings of the 32nd Annual Symposium on Computational Geometry (SoCG'16).}


\newcommand{\neno}{\bibitem[Ne91]{Ne91} \emph{S. Negami.} Ramsey theorems for knots, links and spatial graphs,
Trans. Amer. Math. Soc., 324 (1991), 527--541.}



\newcommand{\nkon}{\bibitem[NKS]{NKS} * \emph{L. T. Nguyen, J. Kim, B. Shim.}
Low-Rank Matrix Completion: A Contemporary Survey. arXiv:1907.11705.}

\newcommand{\noss}{\bibitem[No76]{No76} * \emph{С. П. Новиков.} Топология-1. М.: Наука, 1976. (Итоги науки и техники. ВИНИТИ. Современные проблемы математики. Основные направления, 12).}

\newcommand{\nszn}{\bibitem[NS09]{NS09} \emph{I. Novik and E. Swartz,} Socles of Buchsbaum modules, complexes and posets, Adv. Math. 222 (2009), 2059-2084.}

\newcommand{\nwns}{\bibitem[NW97]{NW97} \emph{A. Nabutovsky, S. Weinberger}. Algorithmic aspects of homeomorphism problems. arXiv:math/9707232.}


\newcommand{\omoe}{\bibitem[Om18]{Om18} * \emph{А. Омельченко,} Теория графов. М.: МЦНМО, 2018.}

\newcommand{\orszo}{\bibitem[ORS]{ORS} \emph{A. Onischenko, D. Repov\v s and A. Skopenkov.}
Resolutions of 2-polyhedra by fake surfaces and embeddings into $\R^4$, Contemp. Math.  288 (2001) 396--400.}

\newcommand{\ossf}{\bibitem[OS74]{OS74} \emph{R. P. Osborne and R. S. Stevens.} Group presentations
corresponding to spines of 3-manifolds, I, Amer. J.~Math. 1974. 96. P.~454-471; II, Amer. J.~Math. 1977. 234.
P.~213-243; III, Amer. J.~Math. 1977. 234 P.~245-251.}


\newcommand{\oz}{\bibitem[Oz]{Oz} \emph{M. \"Ozaydin,} Equivariant maps for the symmetric group, unpublished,
\url{http://minds.wisconsin.edu/handle/1793/63829}.}

\newcommand{\panof}{\bibitem[Pan15]{Pan15} \emph{K. Panagiotis.} A note on the topology of irreducible $SO(3)$-manifolds, 	arXiv:1508.06150.}

\newcommand{\paof}{\bibitem[Pa15]{Pa15} \emph{S. Parsa,} On links of vertices in simplicial $d$-complexes embeddable in the Euclidean $2d$-space, Discrete Comput. Geom. 59:3 (2018), 663--679.
This is arXiv:1512.05164v4 up to numbering of sections, theorems etc.; we refer to numbering in arxiv version.
Correction: Discrete Comput. Geom. 64:3 (2020) 227--228.}

\newcommand{\paoe}{\bibitem[Pa18]{Pa18} \emph{S. Parsa,} On links of vertices in simplicial $d$-complexes
embeddable in the euclidean $2d$-space, arXiv:1512.05164v6.}

\newcommand{\patz}{\bibitem[Pa20]{Pa20} \emph{S. Parsa,} On links of vertices in simplicial $d$-complexes
embeddable in the euclidean $2d$-space, arXiv:1512.05164v8.}


\newcommand{\patzl}{\bibitem[Pa20]{Pa20} \emph{S. Parsa,}
Correction to: On the Links of Vertices in Simplicial $d$-Complexes Embeddable in the Euclidean $2d$-Space,
Discrete Comput. Geom. 64:3 (2020) 227--228.}

\newcommand{\patza}{\bibitem[Pa20]{Pa20} \emph{S. Parsa,} On the Smith classes, the van Kampen obstruction and embeddability of $[3]*K$, arXiv:2001.06478.}

\newcommand{\patzb}{\bibitem[Pa20b]{Pa20b} \emph{S. Parsa,} On the embeddability of $[3]*K$, arXiv:2001.06506.}

\newcommand{\pato}{\bibitem[Pa21]{Pa21} \emph{S. Parsa,} Instability of the Smith index under joins and applications to embeddability, Trans. Amer. Math. Soc. 375 (2022), 7149--7185, arXiv:2103.02563.}

\newcommand{\pak}{\bibitem[Pa]{Pa} * \emph{I. Pak}, Lectures on Discrete and Polyhedral Geometry, \url{http://www.math.ucla.edu/~pak/geompol8.pdf}.}

\newcommand{\peze}{\bibitem[Pe08]{Pe08} \emph{Д. Пермяков.} Классификация погружений графов в плоскость,
Вестник МГУ, сер.1, 2008, N5, 55-56.}

\newcommand{\peos}{\bibitem[Pe16]{Pe16} \emph{Д. Пермяков.} Матем. сб., 207:6 (2016),  93--112.}

\newcommand{\pest}{\bibitem[Pe72]{Pe72} * \emph{B. B. Peterson.} The Geometry of Radon's Theorem, Amer. Math. Monthly 79 (1972), 949-963.}


\newcommand{\prnf}{\bibitem[Pr95]{Pr95} * \emph{V. V. Prasolov.} Intuitive topology. Amer. Math. Soc., Providence, R.I., 1995.}

\newcommand{\prnfr}{\bibitem[Pr95]{Pr95} * \emph{В. В. Прасолов.} Наглядная топология. М.: МЦНМО, 1995.}


\newcommand{\przs}{\bibitem[Pr06]{Pr06} * \emph{V. V. Prasolov.}
Elements of Combinatorial and Differential Topology, 2006, GSM 74, Amer. Math. Soc., Providence, RI.}

\newcommand{\przsru}{\bibitem[Pr04]{Pr04} * \emph{В. В. Прасолов.}
Элементы комбинаторной и дифференциальной топологии. М.: МЦНМО, 2004. \url{http://www.mccme.ru/prasolov}.}

\newcommand{\przse}{\bibitem[Pr07]{Pr07} * \emph{V. V. Prasolov.} Elements of homology theory. 2007, GSM 74, Amer. Math. Soc., Providence, RI.}


\newcommand{\przseru}{\bibitem[Pr06]{Pr06} * \emph{В. В. Прасолов.} Элементы теории гомологий. М.: МЦНМО, 2006.}


\newcommand{\psns}{\bibitem[PS96]{PS96} * \emph{V. V. Prasolov, A. B. Sossinsky } Knots, Links, Braids, and 3-manifolds. Amer. Math. Soc. Publ., Providence, R.I., 1996.}


\newcommand{\pszf}{\bibitem[PS05]{PS05} * \emph{В. В. Прасолов и М. Б. Скопенков.}
Рамсеевская теория зацеплений, Мат. Просвещение. 2005. 9. С.~108--115.}

\newcommand{\pszfen}{\bibitem[PS05]{PS05} * \emph{V. V. Prasolov and M.B. Skopenkov.}
Ramsey link theory, Mat, Prosvescheniye, 9 (2005), 108--115.}

\newcommand{\psoo}{\bibitem[PS11]{PS11} \emph{Y. Ponty and C. Saule.} A combinatorial framework for designing (pseudoknotted) RNA algorithms, Proc. of the 11th Intern. Workshop on Algorithms in Bioinformatics, WABI'11, 250--269.}


\newcommand{\pstz}{\bibitem[PS20]{PS20} \emph{S. Parsa and A. Skopenkov.} On embeddability of joins and their `factors', Topol. Appl., to appear, arXiv:2003.12285.}


\newcommand{\psszn}{\bibitem[PSS]{PSS} \emph{M. J. Pelsmajer, M. Schaefer and D. Stasi.} Strong Hanani-Tutte on the projective plane. SIAM J. Discrete Math., 23:3 (2009) 1317--1323.}

\newcommand{\psszs}{\bibitem[PSS]{PSS} \emph{M. J. Pelsmajer, M. Schaefer, and D. \v Stefankovi\v c.}
Removing even crossings. J. Combin. Theory Ser. B, 97(4):489–500, 2007.}

\newcommand{\pton}{\bibitem[PT19]{PT19} \emph{P. Pat\'ak and M. Tancer.} Embeddings of $k$-complexes into $2k$-manifolds. arXiv:1904.02404.}

\newcommand{\pw}{\bibitem[PW]{PW} \emph{I. Pak, S. Wilson}, G\lowercase{EOMETRIC REALIZATIONS OF POLYHEDRAL COMPLEXES}, \url{http://www.math.ucla.edu/~pak/papers/Fary-full31.pdf}.}


\newcommand{\razf}{\bibitem[RA05]{RA05} * \emph{J. L. Ram\'irez Alfons\'in.} Knots and links in spatial graphs: a survey. Discrete Math., 302 (2005), 225--242.}

\newcommand{\rep}{\bibitem[Rep]{Rep} Referee's report on the paper ``Some `converses' to intrinsic linking theorems', \url{https://www.mccme.ru/circles/oim/materials/ksreport.pdf}}

\newcommand{\rnoo}{\bibitem[RN11]{RN11} * \emph{R. L. Ricca, B. Nipoti.} Gauss' linking number revisited.
J. of Knot Theory and Its Ramif. 20:10 (2011) 1325--1343. \url{https://www.maths.ed.ac.uk/~v1ranick/papers/ricca.pdf} .}

\newcommand{\rrstz}{\bibitem[RRS]{RRS} * \emph{V. Retinskiy, A. Ryabichev and A. Skopenkov.}
Motivated exposition of the proof of the Tverberg Theorem (in Russian).
Mat. Prosveschenie, 27 (2021), 166--169. arXiv:2008.08361.}

\newcommand{\rssec}{\bibitem[RS68]{RS68} \emph{C. P. Rourke and B. J. Sanderson,} Block bundles II, Ann. of Math. (2), 87 (1968) 431--483.}

\newcommand{\rsst}{\bibitem[RS72]{RS72} * \emph{C. P. Rourke and B. J. Sanderson,}
\newblock Introduction to Piecewise-Linear Topology,
\newblock \emph{Ergebn.\ der Math.} 69, Springer-Verlag, Berlin, 1972.}

\newcommand{\rsstr}{\bibitem[RS72]{RS72} * \emph{К. П. Рурк и Б. Дж. Сандерсон.} Введение в кусочно-линейную топологию, Москва. Мир. 1974.}

\newcommand{\rsns}{\bibitem[RS96]{RS96} * \emph{D. Repov\v s and A. B. Skopenkov.}
Embeddability and isotopy of polyhedra in Euclidean spaces,
Proc. of the Steklov Inst. Math. 1996. 212. P.~173-188.}

\newcommand{\rsne}{\bibitem[RS98]{RS98} \emph{D. Repov\v s and A. B. Skopenkov.}
A deleted product criterion for approximability of a map by embeddings, Topol. Appl. 1998. 87 P.~1-19.}

\newcommand{\rsnn}{\bibitem[RS99]{RS99} * \emph{D. Repov\v s and A. B. Skopenkov.} New results on embeddings of polyhedra and manifolds into Euclidean spaces,
Russ. Math. Surv. 54:6 (1999), 1149--1196.}


\newcommand{\rsnnd}{\bibitem[RS99']{RS99'} * \emph{Д. Реповш и А. Скопенков.}
Кольца Борромео и препятствия к вложимости, Труды МИРАН. 1999. 225. С.~331-338.}

\newcommand{\rszz}{\bibitem[RS00]{RS00} \emph{D. Repov\v s and A. Skopenkov.} Cell-like resolutions of polyhedra by special ones,  Colloq. Math. 2000. 86:2. P. 231--237.}

\newcommand{\rszzd}{\bibitem[RS00']{RS00'} * \emph{Д. Реповш и А. Скопенков.} Характеристические классы для начинающих, Мат. Просвещение. 2000. 4. С.~151-176.}

\newcommand{\rszo}{\bibitem[RS01]{RS01} \emph{D. Repovs and A. Skopenkov.} On contractible $n$-dimensional compacta, non-embeddable into $\R^{2n}$, Proc. Amer. Math. Soc. 129 (2001) 627--628.}

\newcommand{\rszt}{\bibitem[RS02]{RS02} * \emph{Д. Реповш и А. Скопенков.} Теория препятствий для начинающих,
Мат. Просвещение. 2002. 6. C.~60-77.}

\newcommand{\rszf}{\bibitem[RS04]{RS04} \emph{N. Robertson and P. Seymour.} Graph Minors. XX. Wagner's conjecture, J. of Comb. Theory, B, 92:2 (2004) 325--357.}

\newcommand{\rssnf}{\bibitem[RSS]{RSS95} \emph{D. Repov\v s, A. B. Skopenkov  and E. V. \v S\v cepin.}
On uncountable collections of continua and their span, Colloq. Math. 1995. 69:2. P.~289-296.}

\newcommand{\rssnfd}{\bibitem[RSS']{RSS95'} \emph{D. Repov\v s, A. B. Skopenkov and E. V \v S\v cepin.}
On embeddability of $X\times I$ into Euclidean space, Houston J.~Math. 1995. 21. P.~199-204.}

\newcommand{\rssz}{\bibitem[RSS+]{RSSZ} * \emph{A. Rukhovich, A. Skopenkov, M. Skopenkov, A. Zimin},
Realizability of hypergraphs, \url{https://www.turgor.ru/lktg/2013/1/1-1en.pdf} .}


\newcommand{\rstnt}{\bibitem[RST']{RST93} \emph{N. Robertson, P. Seymour and R. Thomas}, Linkless embeddings of graphs in 3-space, Bull. of the Amer. Math. Soc., 21 (1993) 84--89.}

\newcommand{\rstno}{\bibitem[RST]{RST91} * \emph{N. Robertson, P. Seymour and R. Thomas}, A survey of
linkless embeddings, Graph Structure Theory (Seattle, WA, 1991), Contemp. Math. 147, (1993) 125--136.}


\newcommand{\rwzl}{\bibitem[RWZ+]{RWZ+} \emph{Y. Ren, C. Wen, S. Zhen, N. Lei, F. Luo, D.X. Gu},
Characteristic class of isotopy for surfaces, J. Syst. Sci. Complex. 33 (2020) 2139--2156.}


\newcommand{\saeo}{\bibitem[Sa81]{Sa81} \emph{H. Sachs.} On spatial representation of finite graphs,
in: Finite and infinite sets (Eger, 1981), 649--662, Colloq. Math. Soc. Janos Bolyai, 37, North-Holland, Amsterdam, 1984.}

\newcommand{\sano}{\bibitem[Sa91]{Sa91} \emph{K. S. Sarkaria.}
A one-dimensional Whitney trick and Kuratowski's graph planarity criterion, Israel J.~Math. 73 (1991), 79--89.
\url{http://kssarkaria.org/docs/One-dimensional.pdf}.}

\newcommand{\sanov}{\bibitem[Sa91g]{Sa91g} \emph{K. S. Sarkaria.} A generalized Van Kampen-Flores theorem, Proc. Amer. Math. Soc. 111 (1991), 559--565.}

\newcommand{\sant}{\bibitem[Sa92]{Sa92} \emph{K. S. Sarkaria.} Tverberg’s theorem via number fields. Israel J. Math., 79:317–320, 1992.}

\newcommand{\sann}{\bibitem[Sa99]{Sa99} O. Saeki {\em On punctured 3-manifolds in 5-sphere}, Hiroshima Math. J. 29 (1999) 255--272.}

\newcommand{\sazz}{\bibitem[Sa00]{Sa00} \emph{K. S. Sarkaria.} Tverberg partitions and Borsuk-Ulam theorems. Pacific J. Math., 196:1 (2000) 231--241.}

\newcommand{\sczf}{\bibitem[Sc04]{Sc04} \emph{T. Sch\"oneborn.} On the Topological Tverberg Theorem, arXiv:math/0405393.}


\newcommand{\scot}{\bibitem[Sc13]{Sc13} * \emph{M. Schaefer.} Hanani-Tutte and related results. In Geometry --- intuitive, discrete, and convex, Bolyai Soc. Math. Stud., 24 (2013), 259--299.
\url{http://ovid.cs.depaul.edu/documents/htsurvey.pdf} }


\newcommand{\sctz}{\bibitem[Sc20]{Sc20} \emph{M. Schaefer.} The Graph Crossing Number and
its Variants: A Survey. The Electr. J. of Comb. (2020), DS21, \url{https://www.combinatorics.org/files/Surveys/ds21/ds21v5-2020.pdf}}


\newcommand{\scef}{\bibitem[Sc84]{Sc84} \emph{E.~V.~\v S\v cepin.} Soft mappings of manifolds, Russian Math. Surveys, 39:5 (1984).}

\newcommand{\shfs}{\bibitem[Sh57]{Sh57} \emph{A. Shapiro,} Obstructions to the embedding of a complex in a Euclidean space, I, The first obstruction, Ann. Math. 66 (1957), 256--269.}


\newcommand{\shen}{\bibitem[Sh89]{Sh89} * \emph{Ю. А. Шашкин,} Неподвижные точки, М., Наука, 1989.}

\newcommand{\shoe}{\bibitem[Sh18]{Sh18} * \emph{S. Shlosman},  Topological Tverberg Theorem: the proofs and the counterexamples, Russian Math. Surveys, 73:2 (2018), 175–182. arXiv:1804.03120.}

\newcommand{\sisn}{\bibitem[Si69]{Si69} \emph{K. Sieklucki.} Realization of mappings, Fund. Math. 1969. 65. P.~325-343.}

\newcommand{\sios}{\bibitem[Si16]{Si16} \emph{S. Simon,} Average-Value Tverberg Partitions via Finite Fourier Analysis, Israel J. Math., 216 (2016), 891-904, arXiv:1501.04612.}



\newcommand{\sknf}{\bibitem[Sk94]{Sk94} \emph{А. Скопенков.} Геометрическое доказательство теоремы
Нойвирта об утолщаемости 2-мерных полиэдров, Math. Notes. 1995. 58:5. P.~1244-1247.}


\newcommand{\skne}{\bibitem[Sk98]{Sk98} \emph{A. B. Skopenkov.} On the deleted product criterion for embeddability in $\R^m$, Proc. Amer. Math. Soc. 1998. 126:8. P.~2467-2476.}

\newcommand{\skzz}{\bibitem[Sk00]{Sk00} \emph{A. Skopenkov,} On the generalized Massey--Rolfsen invariant for link maps, Fund. Math. 165 (2000), 1--15.}

\newcommand{\skzt}{\bibitem[Sk02]{Sk02} \emph{A. Skopenkov,} On the Haefliger-Hirsch-Wu invariants for embeddings and immersions, Comment. Math. Helv. 77 (2002), 78--124.}

\newcommand{\skzth}{\bibitem[Sk03]{Sk03} \emph{M. Skopenkov,} Embedding products of graphs into Euclidean spaces,
Fund. Math. 179 (2003),~191--198, arXiv:0808.1199.}

\newcommand{\skzthd}{\bibitem[Sk03']{Sk03'} \emph{M. Skopenkov,} On approximability by embeddings of cycles in the plane, Topol. Appl. 134 (2003),~1--22, arXiv:0808.1187.}

\newcommand{\skzf}{\bibitem[Sk05]{Sk05} * \emph{A. Skopenkov,}
On the Kuratowski graph planarity criterion, Mat. Prosveschenie, 9 (2005), 116-128. arXiv:0802.3820.}


\newcommand{\skzs}{\bibitem[Sk05i]{Sk05i} \emph{A. Skopenkov,} A new invariant and parametric connected sum of embeddings, Fund. Math. 197 (2007) 253--269. arxiv:math/0509621.}

\newcommand{\skzei}{\bibitem[Sk05]{Sk05} \emph{A.  Skopenkov,} A classification of smooth embeddings of
4-manifolds in 7-space, I, Topol. Appl., 157 (2010) 2094--2110. arXiv:math/0512594.}

\newcommand{\skze}{\bibitem[Sk06]{Sk06} * \emph{A. Skopenkov,} Embedding and knotting of manifolds in Euclidean spaces, London Math. Soc. Lect. Notes, 347 (2008) 248--342. arXiv:math/0604045.}

\newcommand{\skzsi}{\bibitem[Sk06']{Sk06'} \emph{A. Skopenkov,} A classification of smooth embeddings of 3-manifolds in 6-space, Math. Zeitschrift, 260:3 (2008) 647--672. arxiv:math/0603429.}

\newcommand{\skozp}{\bibitem[Sk08]{Sk08} \emph{A.  Skopenkov,} Embeddings of $k$-connected $n$-manifolds into
$\R^{2n-k-1}$. arxiv:math/0812.0263; earlier version published in Proc. Amer. Math. Soc., 138 (2010) 3377--3389.}

\newcommand{\skoz}{\bibitem[Sk10]{Sk10} * \emph{А. Скопенков,} Вложения в плоскость графов с вершинами степени 4,
Мат. Просвещение, 21 (2017), arXiv:1008.4940.}

\newcommand{\skoo}{\bibitem[Sk11]{Sk11} \emph{M. Skopenkov,} When is the set of embeddings finite up to isotopy? Intern. J. Math. 26:7 (2015), 28 pp. arXiv:1106.1878.}

\newcommand{\sks}{\bibitem[Sk14]{Sk14} * \emph{A. Skopenkov,} Realizability of hypergraphs and intrinsic linking  theory, arXiv:1402.0658.}


\newcommand{\skof}{\bibitem[Sk15]{Sk15} * \emph{А. Скопенков,} Алгебраическая топология с геометрической точки зрения, Москва, МЦНМО, 2015 (1е издание).}

\newcommand{\skofe}{\bibitem[Sk15]{Sk15} * \emph{A. Skopenkov,} Algebraic Topology From Geometric Viewpoint (in Russian), MCCME, Moscow, 2015 (1st edition). }

\newcommand{\skofel}{\bibitem[Sk15e]{Sk15e} * \emph{А. Скопенков,} Алгебраическая топология
с геометрической точки зрения, эл. версия, \url{http://www.mccme.ru/circles/oim/home/combtop13.htm\#photo}}


\newcommand{\skotzr}{\bibitem[Sk20]{Sk20} * \emph{А. Скопенков,} Алгебраическая топология с геометрической точки зрения, Москва, МЦНМО, 2020 (2е издание).
Часть книги: \url{http://www.mccme.ru/circles/oim/obstruct.pdf}}

\newcommand{\skotz}{\bibitem[Sk20]{Sk20} * \emph{A. Skopenkov,} Algebraic Topology From Geometric Standpoint (in Russian), MCCME, Moscow, 2020 (2nd edition).
Part of the book: \url{http://www.mccme.ru/circles/oim/obstruct.pdf} .
Part of the English translation: \url{https://www.mccme.ru/circles/oim/obstructeng.pdf}. }



\newcommand{\skofp}{\bibitem[Sk15]{Sk15} \emph{A. Skopenkov,} Classification of knotted tori,
Proc. A of the Royal Soc. of Edinburgh, 150:2 (2020), 549-567. Full version: arXiv:1502.04470.}


\newcommand{\skos}{\bibitem[Sk16]{Sk16} * \emph{A. Skopenkov,} A user's guide to the topological Tverberg Conjecture, arXiv:1605.05141v4. Abridged earlier published version: Russian Math. Surveys, 73:2 (2018), 323--353.}



\newcommand{\skosd}{\bibitem[Sk16']{Sk16'} * \emph{A. Skopenkov,} Stability of intersections of graphs in the plane and the van Kampen obstruction, Topol. Appl. 240(2018) 259--269, arXiv:1609.03727.}


\newcommand{\skosc}{\bibitem[Sk16c]{Sk16c} * \emph{A. Skopenkov,}  Embeddings in Euclidean space: an introduction to their classification, to appear in Boll. Man. Atl. http://www.map.mpim-bonn.mpg.de/Embeddings\_in\_Euclidean\_space:\_an\_introduction\_to\_their\_classification}

\newcommand{\skosie}{\bibitem[Sk16e]{Sk16e} * \emph{A. Skopenkov,} Embeddings just below the stable range: classification, to appear in Boll. Man. Atl.
http://www.map.mpim-bonn.mpg.de/Embeddings\_just\_below\_the\_stable\_range:\_classification}

\newcommand{\skost}{\bibitem[Sk16t]{Sk16t} * \emph{A. Skopenkov,} 3-manifolds in 6-space, to appear in Boll. Man. Atl.
http://www.map.mpim-bonn.mpg.de/3-manifolds\_in\_6-space}

\newcommand{\skosf}{\bibitem[Sk16f]{Sk16f} * \emph{A. Skopenkov,} 4-manifolds in 7-space, to appear in Boll. Man. Atl. http://www.map.mpim-bonn.mpg.de/4-manifolds\_in\_7-space}

\newcommand{\skosh}{\bibitem[Sk16h]{Sk16h} * \emph{A. Skopenkov,} High codimension links, to appear in Boll. Man. Atl. \url{http://www.map.mpim-bonn.mpg.de/High_codimension_links}.}

\newcommand{\skosi}{\bibitem[Sk16i]{Sk16i} * \emph{A. Skopenkov,} Isotopy, submitted to Boll. Man. Atl.
\url{http://www.map.mpim-bonn.mpg.de/Isotopy}.}

\newcommand{\skose}{\bibitem[Sk17]{Sk17} \emph{A. Skopenkov,}
Eliminating higher-multiplicity intersections in the metastable dimension range. arXiv:1704.00143.}

\newcommand{\skosed}{\bibitem[Sk17v]{Sk17v} * \emph{A. Skopenkov,}
On van Kampen-Flores, Conway-Gordon-Sachs and Radon theorems,  arXiv:1704.00300.}

\newcommand{\sk}{\bibitem[Sk17o]{Sk17o} \emph{A. Skopenkov,} On the metastable Mabillard-Wagner conjecture.  arXiv:1702.04259.}

\newcommand{\skmos}{\bibitem[Sk17d]{Sk17d} \emph{M. Skopenkov}. Discrete field theory: symmetries and conservation laws, arXiv:1709.04788.}

\newcommand{\skoe}{\bibitem[Sk18]{Sk18} * \emph{A. Skopenkov.} Invariants of graph drawings in the plane.
Arnold Math. J., 6 (2020) 21--55; full version: arXiv:1805.10237.}


\newcommand{\skoeo}{\bibitem[Sk18o]{Sk18o} * \emph{A. Skopenkov.} A short exposition of S. Parsa's theorems on intrinsic linking and non-realizability. Discr. Comp. Geom. 65:2 (2021), 584--585; full version:  arXiv:1808.08363.}


\newcommand{\skona}{\bibitem[Sk19]{Sk19} * \emph{A. Skopenkov,} A short exposition of the Levine-Lidman example of spineless 4-manifolds, arXiv:1911.07330.}

\newcommand{\sktze}{\bibitem[Sk21m]{Sk21m} * \emph{A. Skopenkov.} Mathematics via Problems. Part 1: Algebra. Amer. Math. Soc., Providence, 2021. Preliminary version: \url{https://www.mccme.ru/circles/oim/algebra_eng.pdf}}

\newcommand{\sktz}{\bibitem[Sk20u]{Sk20u} * \emph{A. Skopenkov.} A user's guide to basic knot and link theory,
in: Topology, Geometry, and Dynamics, Contemporary Mathematics, vol. 772, Amer. Math. Soc., Providence, RI, 2021, pp. 281--309.
Russian version: Mat. Prosveschenie 27 (2021), 128--165. arXiv:2001.01472.}

\newcommand{\sktzo}{\bibitem[Sk20o]{Sk20o} \emph{A. Skopenkov.} On some results of S. Abramyan and T. Panov, arXiv:2005.11152.}

\newcommand{\sktzr}{\bibitem[Sk20e]{Sk20e} * \emph{A. Skopenkov.}
Extendability of simplicial maps is undecidable, Discr. Comp. Geom., 69:1 (2022), 1--10, arXiv:2008.00492.}


\newcommand{\sktzd}{\bibitem[Sk21d]{Sk21d} * \emph{A. Skopenkov.}
On different reliability standards in current mathematical research, arXiv:2101.03745.
More often updated version: \url{https://www.mccme.ru/circles/oim/rese_inte.pdf}.}

\newcommand{\sktt}{\bibitem[Sk22]{Sk22} * \emph{A. Skopenkov.} Invariants of embeddings of 2-surfaces in 3-space,
arXiv:2201.10944.}

\newcommand{\skd}{\bibitem[Sk]{Sk} * \emph{А. Скопенков.} Алгебраическая топология с алгоритмической точки зрения, 
\url{http://www.mccme.ru/circles/oim/algor.pdf}.}

\newcommand{\skde}{\bibitem[Sk]{Sk} * \emph{A. Skopenkov.} Algebraic Topology From Algorithmic Standpoint, draft of a book, mostly in Russian,
\url{http://www.mccme.ru/circles/oim/algor.pdf}.}


\newcommand{\skon}{\bibitem[Skw]{Skw} * \emph{A. Skopenkov.} Whitney trick for eliminating multiple intersections, slides for talks at St Petersburg, Brno, Kiev, Moscow,  \url{https://www.mccme.ru/circles/oim/eliminat_talk.pdf}.}

\newcommand{\skl}{\bibitem[EEF]{EEF} * {\it Proposed by D. Eliseev, A. Enne, M. Fedorov, A. Glebov, N. Khoroshavkina, E. Morozov, A. Skopenkov, R. \v Zivaljevi\'c.}
A user's guide to knot and link theory, \url{https://www.turgor.ru/lktg/2019/3} .}

\newcommand{\skr}{\bibitem[Skr]{Skr} * \emph{A. Skopenkov.} Realizability of hypergraphs, slides for talks,  \url{https://www.mccme.ru/circles/oim/algor1_beamer.pdf}.}

\newcommand{\skt}{\bibitem[Skt]{Skt} * \emph{A. Skopenkov.} Transparent anonymous peer review,
\linebreak
\url{https://www.mccme.ru/circles/oim/home/transp_peer_review.htm} .}

\newcommand{\rslktg}{\bibitem[KRR+]{RRSl} * Towards higher-dimensional combinatorial geometry, presented by
E. Kogan, V. Retinskiy, E. Riabov and A. Skopenkov, \url{https://www.mccme.ru/circles/oim/multicomb.pdf} .}


\newcommand{\sm}{\bibitem[Sm]{Sm} S. Smirnov.}

\newcommand{\sper}{\bibitem[Sp]{Sp} * Sperner's lemma defeats the rental harmony problem, \url{https://www.youtube.com/watch?v=7s-YM-kcKME}.}

\newcommand{\sset}{\bibitem[SS83]{SS83} \emph{Е. В. Щепин, М. А. Штанько.} Спектральный критерий вложимости компактов в евклидовы пространства, Труды Ленинградской Международной Топологической конференции. Л.: Наука, 1983. С.~135-142.}

\newcommand{\ssnt}{\bibitem[SS92]{SS92} \emph{J.~Segal and S.~Spie\.z.} Quasi embeddings and embeddings of polyhedra in $\R^m$,  Topol. Appl., 45 (1992) 275--282.}

\newcommand{\sszt}{\bibitem[SS03]{SS03} \emph{F. W. Simmons and F. E. Su.}
Consensus-halving via theorems of Borsuk-Ulam and Tucker, Math. Social Sciences 45 (2003) 15–25. \url{https://www.math.hmc.edu/~su/papers.dir/tucker.pdf}.}

\newcommand{\ssot}{\bibitem[SS13]{SS13} \emph{M. Schaefer and D. Stefankovi\v c.} Block additivity of $\Z_2$-embeddings. In Graph drawing, volume 8242 of Lecture Notes in Comput. Sci., 185--195.
Springer, Cham, 2013. \url{http://ovid.cs.depaul.edu/documents/genus.pdf}}

\newcommand{\sssne}{\bibitem[SSS]{SSS} \emph{J. Segal, A. Skopenkov and S. Spie\. z.}
Embeddings of polyhedra in $\R^m$ and the deleted product obstruction, Topol. Appl. 1998. 85. P.~225-234.}

\newcommand{\sstnf}{\bibitem[SST95]{SST95} \emph{R. S. Simon, S. Spie\. z and H. Toru\'nczyk.}
T\lowercase{HE EXISTENCE OF EQUILIBRIA IN CERTAIN GAMES, SEPARATION FOR FAMILIES OF CONVEX FUNCTIONS
AND A THEOREM OF BORSUK-ULAM TYPE}, Israel J. Math 92 (1995) 1--21.}

\newcommand{\sstzt}{\bibitem[SST02]{SST02} \emph{R. S. Simon, S. Spie\. z and H. Toru\'nczyk.}
E\lowercase{QUILIBRIUM EXISTENCE AND TOPOLOGY IN SOME REPEATED GAMES WITH INCOMPLETE INFORMATION},
Trans. Amer. Math. Soc. 354:12 (2002) 5005-5026.}

\newcommand{\stez}{\bibitem[ST80]{ST80} * {\it H.~Seifert and W.~Threlfall.}
A textbook of topology, v~89 of {\em Pure and Applied Mathematics}.
Academic Press, New York-London, 1980.}


\newcommand{\stzs}{\bibitem[ST07]{ST07} * \emph{А. Скопенков и А. Телишев.}
И вновь о критерии Куратовского планарности графов, Мат. Просвещение, 11 (2007), 159--160.}

\newcommand{\stzse}{\bibitem[ST07]{ST07} * \emph{A. Skopenkov and A. Telishev}, Once again on the Kuratowski graph planarity criterion, Mat. Prosveschenie, 11 (2007), 159-160. arXiv:0802.3820.}

\newcommand{\stos}{\bibitem[ST17]{ST17} \emph{A. Skopenkov  and M. Tancer,}
Hardness of almost embedding simplicial complexes in $\R^d$, Discr. Comp. Geom., 61:2 (2019), 452--463. arXiv:1703.06305.}

\newcommand{\stno}{\bibitem[ST91]{ST91} \emph{S.~Spie\. z and H.~Toru\'nczyk}, Moving compacta in $\R^m$ apart,
Topol. Appl. 41 (1991), 193--204.}

\newcommand{\stwh}{\bibitem[SW]{SW} * \url{http://www.map.mpim-bonn.mpg.de/Stiefel-Whitney_characteristic_classes}}

\newcommand{\sz}{\bibitem[SZ05]{SZ} \emph{T. Sch\"oneborn and G. Ziegler}, The Topological Tverberg Theorem and Winding Numbers, J. Comb. Theory, Ser. A, 112:1 (2005) 82--104, arXiv:math/0409081.}

\newcommand{\szno}{\bibitem[Sz91]{Sz91} \emph{A.~Sz\"ucs,} On the cobordism groups of immersions and embeddings,
Math. Proc. Camb. Phil. Soc., 109 (1991) 343--349.}


\newcommand{\ta}{\bibitem[Ta]{Ta} * Handbook of Graph Drawing and Visualization. ed. by R. Tamassia, CRC Press, 2016.}


\newcommand{\tazz}{\bibitem[Ta00]{Ta00} \emph{K. Taniyama,} Higher dimensional links in a simplicial complex embedded in a sphere, Pacific Jour. of Math. 194:2 (2000), 465-467.}

\newcommand{\theo}{\bibitem[Th81]{Th81} * \emph{C.~Thomassen,} Kuratowski's theorem, J.~Graph. Theory 5 (1981), 225--242.}

\newcommand{\tooo}{\bibitem[To11]{To11} \emph{Tonkonog D.} Embedding 3-manifolds with boundary into closed 3-manifolds, Topol. Appl. 158 (2011), 1157-1162. arXiv:1003.3029.}


\newcommand{\tsbzf}{\bibitem[TSB]{TSB} \emph{D. M. Thilikos, M. Serna and H. L. Bodlaender},
Cutwidth I: A linear time fixed parameter algorithm, J. of Algorithms, 56:1 (2005), 1--24.}


\newcommand{\tsbzfd}{\bibitem[TSB05']{TSB05'} \emph{D. M. Thilikos, M. Serna and H. L. Bodlaender},
Cutwidth II: , J. of Algorithms, 56:1 (2005), 25--49.}



\newcommand{\umse}{\bibitem[Um78]{Um78} \emph{B. Ummel.} The product of nonplanar complexes does not imbed in 4-space, Trans. Amer. Math. Soc., 242 (1978) 319--328.}




\newcommand{\val}{\bibitem[Val]{Val} * \url{https://en.wikipedia.org/wiki/Valknut}}


\newcommand{\vi}{\bibitem[Vi]{Vi} * \emph{O. Viro.}
Some integral calculus based on Euler characteristic, Lect. Notes in Math. 1346.}

\newcommand{\vizt}{\bibitem[Vi02]{Vi02} * \emph{Э. Б. Винберг.} Курс алгебры. Москва. Факториал Пресс. 2002.}

\newcommand{\vizteng}{\bibitem[Vi02]{Vi02} * \emph{E. B. Vinberg.} A Course in Algebra. Graduate Studies in Mathematics, vol. 56. 2003.}

\newcommand{\vinhzs}{\bibitem[VINH07]{VINH07} * \emph{О. Я. Виро, О. А. Иванов, Н. Ю. Нецветаев и В. М. Харламов.}
Элементарная топология, МЦНМО. 2007.}

\newcommand{\vktt}{\bibitem[vK32]{vK32} \emph{E.~R.~van~Kampen}, Komplexe in euklidischen R\"aumen, Abh. Math. Sem. Hamburg, 9 (1933) 72--78; Berichtigung dazu, 152--153.
English translation by Tu T$\hat a$m Ngu$\tilde{\hat e}$n-Phan:
\url{https://sites.google.com/site/tutamnguyenphan/van_Kampen.pdf}}

\newcommand{\kafo}{\bibitem[vK41]{vK41} \emph{E. R. van Kampen,} Remark on the address of S. S. Cairns,
in Lectures in Topology, 311--313, University of Michigan Press, Ann Arbor, MI, 1941.}

\newcommand{\vo}{\bibitem[Vo96]{vo96} \emph{A. Yu. Volovikov,} On a topological generalization of the Tverberg theorem. Math. Notes 59:3 (1996), 324--326.}

\newcommand{\vopns}{\bibitem[Vo96v]{Vo96v} \emph{A. Yu. Volovikov,} On the van Kampen-Flores Theorem.
Math. Notes 59:5 (1996), 477--481.}

\newcommand{\vznt}{\bibitem[VZ93]{VZ93} \emph{A. Vu\v ci\'c and R. T. \v Zivaljevi\'c}, Note on a conjecture of Sierksma, Discr. Comput. Geom. 9 (1993), 339-349.}

\newcommand{\vzzn}{\bibitem[VZ09]{VZ09} \emph{S. T. Vre\'cica and R. T. \v Zivaljevi\'c},  Chessboard complexes
indomitable, J. of Comb. Theory, Ser. A 118:7 (2011), 2157--2166. arXiv:0911.3512.}


\newcommand{\walst}{\bibitem[Wa62]{Wa62} \emph{C.~T.~C.~Wall}, Classification of $(n-1)$-connected $2n$-manifolds, Ann. of Math., 75 (1962) 163--189.}


\newcommand{\wallss}{\bibitem[Wa67]{Wa67} \emph{C.~T.~C.~Wall.} Classification problems in differential topology, IV, Thickenings, Topology 1966. 5. P. 73--94.}

\newcommand{\waldss}{\bibitem[Wa67m]{Wa67m} \emph{F. Waldhausen.} Eine Klasse von 3-dimensional Mannigfaltigkeiten, I. Invent. Math. 1967. 3. P.~308-333.}

\newcommand{\walsz}{\bibitem[Wa70]{Wa70} \emph{C. T. C. Wall,} Surgery on compact manifolds,
1970, Academic Press, London.}

\newcommand{\wess}{\bibitem[We67]{We67} \emph{C.~Weber.} Plongements de poly\`edres dans le domain metastable, Comment. Math. Helv. 42 (1967), 1--27.}

\newcommand{\whit}{\bibitem[Wl]{Wl} * \url{https://en.wikipedia.org/wiki/Whitehead_link}}

\newcommand{\winum}{\bibitem[Wn]{Wn} * \url{https://en.wikipedia.org/wiki/Winding_number}}

\newcommand{\wrss}{\bibitem[Wr77]{Wr77} \emph{P. Wright.} Covering 2-dimensional polyhedra by 3-manifolds spines.
Topology. 16 (1977), 435--439.}

\newcommand{\wufe}{\bibitem[Wu58]{Wu58} \emph{W. T. Wu.} On the realization of complexes in a euclidean space (in Chinese): I, Sci Sinica, 7 (1958) 251--297; II, Sci Sinica, 7 (1958) 365--387; III, Sci Sinica, 8 (1959) 133--150.}

 \newcommand{\wufn}{\bibitem[Wu59]{Wu59} \emph{W.~T.~Wu.} On the isotopy of a finite complex in Euclidean space, I, II, Science Record, N.S. 3:8 (1959) 342--347, 348--351.}

\newcommand{\wusf}{\bibitem[Wu65]{Wu65} * \emph{W. T. Wu.} A Theory of Embedding, Immersion and Isotopy of Polytopes in an Euclidean Space. Peking: Science Press, 1965.}


\newcommand{\yann}{\bibitem[Ya99]{Ya99} \emph{Z. Yang.} Computing Equilibria and Fixed Points: The Solution of Nonlinear Inequalities, Kluwer, Springer Science + Business Media, 1990.}


\newcommand{\zesz}{\bibitem[Ze60]{Ze60} \emph{E. C. Zeeman}, Unknotting spheres in five dimensions, Bull. Amer. Math. Soc. 66 (1960) 198.
\linebreak
\url{https://www.ams.org/journals/bull/1960-66-03/S0002-9904-1960-10431-4/S0002-9904-1960-10431-4.pdf}}

\newcommand{\z}{\bibitem[Ze]{Z} * \emph{E. C. Zeeman}, A Brief History of Topology, UC Berkeley, October 27, 1993, On the occasion of Moe Hirsch's 60th birthday, \url{http://zakuski.utsa.edu/~gokhman/ecz/hirsch60.pdf}.}

\newcommand{\zioz}{\bibitem[Zi10]{Zi10} * \emph{D. \v Zivaljevi\'c}, Borromean and Brunnian Rings
\url{http://www.rade-zivaljevic.appspot.com/borromean.html}.}

\newcommand{\zioo}{\bibitem[Zi11]{Zi11} * \emph{G. M. Ziegler}, 3N Colored Points in a Plane, Notices of the Amer. Math. Soc., 58:4 (2011), 550-557.}


\newcommand{\zot}{\bibitem[Zi13]{Z13} \emph{A. Zimin.} Alternative proofs of the Conway-Gordon-Sachs Theorems, arXiv:1311.2882.}


\newcommand{\zss}{\bibitem[ZSS]{ZSS} * Элементы математики в задачах: через олимпиады и кружки к профессии
Сборник под редакцией А. Заславского, А. Скопенкова и М. Скопенкова. Изд-во МЦНМО, 2018.
\url{http://www.mccme.ru/circles/oim/materials/sturm.pdf}.}


\newcommand{\zu}{\bibitem[Zu]{Zu} \emph{J. Zung.} A non-general-position Parity Lemma,
\url{http://www.turgor.ru/lktg/2013/1/parity.pdf}.}








\bkkmzof
\cfsz
\aronly{\grsz}
\gszs
\mazt
\mezs
\mett
\mesczs
\paof
\patza
\pato
\jonly{\pstz}
\skzt
\skze
\skoe
\skoeo
\wess

\end{thebibliography}
\end{document}